\documentclass[a4paper,11pt]{elsarticle}
\usepackage[DIV11]{typearea}
\usepackage{setspace}
\usepackage[T1]{fontenc}
\usepackage[english]{babel}
\usepackage{lmodern}
\usepackage[hidelinks]{hyperref} 
\usepackage{csquotes} 

\usepackage{algorithmicx,algpseudocode}
\usepackage[ruled,vlined,linesnumbered]{algorithm2e}

\usepackage{amsfonts,amssymb,amsmath,amsthm}
\usepackage{graphicx}  
\usepackage{marginnote}
\usepackage[dvipsnames]{xcolor}
\usepackage{enumitem}
\usepackage{tabularx}
\usepackage{subfigure}
\usepackage{mathtools}
\usepackage{bm}

\theoremstyle{definition}
\newtheorem{theorem}{Theorem}[section]

\newtheorem{definition}[theorem]{Definition}
\newtheorem{remark}[theorem]{Remark}
\newtheorem{lemma}[theorem]{Lemma}

\newtheorem{proposition}[theorem]{Proposition}

\newtheorem{example}[theorem]{Example}

\newcommand{\leqoeq}{\left\{\begin{smallmatrix}=\\\leq\end{smallmatrix}\right\}}

\newcommand{\x}{\bm{x}}
\renewcommand{\b}{\bm{b}}
\newcommand{\g}{\bm{g}}
\newcommand{\0}{\bm{0}}
\newcommand{\f}{\bm{f}}
\renewcommand{\u}{\bm{u}}
\renewcommand{\v}{\bm{v}}
\newcommand{\y}{\bm{y}}
\newcommand{\Sys}{\bm{\mathcal{S}}}

\newcommand{\hide}[1]{}

\newcommand{\R}{\mathbb{R}}

\DeclareMathOperator{\Sup}{Sup}
\DeclareMathOperator{\epsSup}{\varepsilon-Sup}
\DeclareMathOperator{\epsjSup}{\varepsilon_\textit{j}-Sup}
\DeclareMathOperator{\wSup}{wSup}
\DeclareMathOperator{\sSup}{sSup}

\DeclareMathOperator*{\vmin}{vmin}
\DeclareMathOperator*{\smin}{smin}

\DeclareMathOperator*{\argmin}{arg\,min}

\DeclareMathOperator*{\pred}{pred}

\usepackage{tikz, pgfplots}
\usetikzlibrary{arrows, shapes.misc, chains, patterns, shapes,backgrounds,fit}
\usetikzlibrary{decorations.pathreplacing}
\usepackage{xifthen}
\newcommand{\precqq}{\begin{matrix}
		\hspace{1mm}\prec{ }\\[-3mm]
		\hspace{1mm}= { }
\end{matrix}}

\tikzset{
	candidat/.style={rectangle, inner sep=0pt, minimum size=0.1cm, draw=gray, fill=gray},
	nds/.style={circle, inner sep=0pt, minimum size=0.12cm, draw=black, fill=black},
	ndns/.style={rectangle, inner sep=0pt, minimum size=0.12cm, draw=black, fill=black},
	test/.style={circle, inner sep=0pt, minimum size=0.12cm, draw=black, fill=black},
}
\pgfplotsset{compat=1.8}

\journal{European Journal of Operational Research}

\begin{document}

\begin{frontmatter}



\title{Introducing Multiobjective Complex Systems}


\author[tuk]{T.~Dietz}
\author[wup]{K.~Klamroth}
\author[wup]{K.~Kraus}
\author[tuk]{S.~Ruzika}
\author[tuk]{L.~E.~Sch\"afer}
\author[wup]{B.~Schulze}
\author[wup]{M.~Stiglmayr}
\author[clem]{M.~M.~Wiecek}

\address[tuk]{Department of Mathematics, Technische Universit\"at Kaiserslautern, Germany}
\address[wup]{School of Mathematics and Natural Sciences, University of Wuppertal, Germany}
\address[clem]{Department of Mathematical Sciences, Clemson University, SC, USA}

\begin{abstract}
This article focuses on the optimization of a complex system which is composed of several subsystems. On the one hand, these subsystems are subject to multiple objectives, local constraints as well as local variables, and they are associated with an own, subsystem-dependent decision maker. On the other hand, these subsystems are interconnected to each other by global variables or linking constraints. Due to these interdependencies, it is in general not possible to simply optimize each subsystem individually to improve the performance of the overall system. This article introduces a formal graph-based representation of such complex systems and generalizes the classical notions of feasibility and optimality to match this complex situation. Moreover, several algorithmic approaches are suggested and analyzed.

\end{abstract}


\end{frontmatter}


\section{Introduction}\label{sec:intro}
\noindent In the modern world, complex systems are pervasive and therefore of high importance to the society. Financial markets, social networks, communication systems, public health providers, cybersecurity systems, global corporations, educational organizations are all examples of complex systems that are composed of multiple but dissimilar parts in the form of subsystems that give rise to the collective behavior of the overall system\footnote{See the homepage of the New England Complex Systems Institute \url{www.necsi.edu/} for a collection of examples of complex systems as well as a colloquial, non-mathematical definition of a complex system \url{www.necsi.edu/guide/study.html}.}. 
The subsystems may be interconnected in a variety of ways and interact with one another. The performance of complex systems may be evaluated by multiple and conflicting criteria that may be different for every subsystem. The subsystems may be represented by models originating from different science or engineering disciplines, which cannot be integrated into an overall model. In the presence of this complexity, an all-in-one (AiO) system representing composition of all subsystems may not exist or may exist only virtually. In fact, there may be more than one way of making up the AiO system. 

In sciences and engineering, the term \enquote{system} is used in a variety of contexts. In mathematical sciences, a system is generally understood as an entity described by input and output relationships. In engineering, the term \enquote{system} has a more specific, application-related meaning. For example, in engineering design, the process of designing a vehicle is a complex system from two different perspectives. It involves interaction among several science or engineering disciplines (e.g.,~aerodynamics, electrical systems, control systems) and therefore
is performed within the field of multidisciplinary design optimization (MDO) that had been developed to address the multidisciplinary complexity of design \citep{Vander2007,MarLam2013}. On the other hand, a vehicle design process is typically hierarchical since it is performed at the system (vehicle) level and the subsystem level where independent teams are responsible for designing subsystems such as tires, engines, batteries, etc.  \citep{Kimetal2003,Kangetal2014}.
In business, a global corporation is a complex system since it may have divisions in various geographical locations or because its activities conducted at the international level may require different strategies from those performed at the local level \citep{Nolan2007}.
In economics (e-commerce) and artificial intelligence, the interaction of several entities with particular interests and goals is analyzed using negotiation analysis \citep{LouWAppr2016}. Each negotiator can be interpreted as a subsystem participating in the overall negotiation process. 
In biology and robotics, the behavior of agents in a multiagent context can be represented as a network-based complex system \citep{JiLiOnre2014}.

The development of mathematical models of complex systems and algorithms for decision making is in concert with other ongoing trends in sciences and engineering. One trend is motivated by the availability of massive data in management, business, and engineering applications along with increasing computing capabilities.  The other results from the intention to provide personalized services to subgroups and individuals but at the same time use mass customization when providing goods to large but diversified populations. In the following, we review modeling and methodological efforts that we find relevant to multiobjective complex systems. Because the literature on 
single-objective settings is vast and deserves a thorough review, it cannot be part of this paper. We therefore only highlight the models and methods and include a representative reference.

\paragraph{Models}
Complex systems can be conveniently modeled as a collection of subsystems represented by optimization problems. The collection may assume two basic structures:  hierarchical (centralized) or nonhierarchical (distributed). In the former, the levels of the hierarchy imply the order of decision making and the direction of the flow of information between the levels. Decision making at a given level may take place after the decision making at the higher level has been accomplished and the information flow has been sent to the lower level 
\citep{ShimHier1981,HaimHier1990,LiWuDesi2014}. A special type of hierarchical modeling is exhibited in bilevel optimization where one level problem is nested within the other \citep{ShimIshBa1997,Bard1998}.
The nonhierarchical structure allows the decision making at a subsystem to be performed independently of the other subsystems with the information being shared between the subsystems as required \citep{StumDece2003,MiguAdec2009,CiucAnin2012}. 

The optimization problems in the collection can be single objective or have multiple objective functions giving rise to multiobjective optimization problems (MOPs). Among many categories of optimization problems, MOPs make up an important class not only due to their special structure that is amenable to analytical derivations and algorithmic developments, but more importantly, due to their capability to model and quantify tradeoffs for informed decision making.

In some studies, the collection assumes the form of a network. In group decision making, \citet{FernAnou2013} represent a complex decision problem as a network of MOPs with each decision maker assigned to a node of the network. \citet{KonnVect2013} uses a network to model spatially distributed elements of complex systems encountered in transportation, communications, and economics. \citet{GW:2015} use a network of MOPs to model multiteam, multidisciplinary, and multiobjective design process.

Some complex systems are modeled as one specially structured MOP. Usually, complex systems occuring in multiparty negotiation processes are modeled as MOPs with one objective function and one feasible set per subsystem, see, e.g.,~\citet{LouWAppr2016}. \citet{diMaMult2017} study spatially distributed storm water harvesting systems, while \citet{NazMMult2016} review multiple objectives for resource management on microgrids. \citet{lewis1997} use a game theory approach to model interactions between the subsystems in MDO.  

The performance of complex system is measured by means of scalar or vector-valued objective functions that act on three types of variables typically associated with the system. The variables associated only with a specific subsystem are referred to as local, the variables associated with two or more subsystems are referred to as global, and the variables modeling the interaction between subsystems are referred to as linking \citep{AonuArec1982}. In some models, global variables play the role of linking variables \citep{FulgDece2007}. The interaction between subsystems can also be modeled more generally by linking functions \citep{TossAugm2008,TossBloc2009,gard:alge:2011}. In some studies, interacting subsystems are called \enquote{interconnected} \citep{FulgDece2007}. The type of the variables plays a fundamental role in the development of decision making methodologies for complex systems.

\paragraph{Methodology} 
Development of decision making methodologies for complex systems may be very difficult because optimal decisions for the subsystems may not be optimal for the AiO system and vice versa. A unique decision optimal for the system may not exist, or if it exists, it may be extremely difficult to be decided upon.  Furthermore, a solution methodology for finding optimal decisions for the AiO system may not exist either, or if it exists, it may be prohibitively expensive due to  difficulties such as heterogeneous functions, integrality of variables, nested problems in a bilevel structure, cost of simulation, etc. 

Due to modeling and methodological challenges, it is of interest to develop distributed decision making methodologies for computing suboptimal decisions for subsystems without ever dealing with the AiO system in its entirety but such that they are suboptimal or optimal to the AiO system. The assumed concept of (sub)optimality may be crucial for the overall success. 

Due to the complexity reflected in the structure of complex systems various solution approaches have been developed to find an optimal solution for the AiO system without dealing with this system in its entirety.  When individually optimizing the subsystems, the most important issues are the treatment of global variables and linking constraints, and coordination of the individual optimization processes among the subsystems to ensure that the AiO system is being optimized.

Global  variables can be treated in three different ways: they can be fixed while the subsystems are optimized, they can be copied to become additional local variables for the subystems at which they are present, or they can be treated as parameters. \citet{AonuArec1982} proposes an iterative process of fixing the global variables and optimizing the subsystems to bring the complex system to optimality. Since the copies must be equal to the original variables, new equality constraints are added to the model. \citet{LeveMult2016} first optimize individual subsystems treating global variables as parameters and then iterate towards an AiO optimal solution using the individual parametric optimal solutions. Linking constraints or the constraints created by the copies are typically relaxed using Lagrangian relaxation \citep{LucaExac1987,TossAugm2008,TossBloc2009,gard:alge:2011} or  penalty methods \citep{FulgDece2007,Kangetal2014}.

The coordination of the individual optimization processes is typically designed in two ways: the systems are directly coordinated \citep{DW:2015} or a master level optimization problem is added \citep{TossAugm2008}. Methods used in the coordination include the block coordinate descent \citep{TossAugm2008,DW:2015}, alternate direction method of multipliers \citep{TossBloc2009}, subgradient optimization \citep{LeveMult2016}, and evolutionary algorithms. When coordination is conducted on a network with subsystems assigned to its nodes, Lagrangian relaxation \citep{GW:2015} or a network equilibrium approach \citep{KonnVect2013} are used. 

A two-level coordination is conducted in multiparty negotiation where the goal is to find an AiO optimal solution without sharing information neither with other negotiators (subsystems) nor with a neutral mediator at the master level. In an iterative procedure the mediator gives a tentative agreement to the negotiators who separately provide their preferences with respect to this suggestion.
In the constraint proposal method each negotiator optimizes on a subset of the decision space which is defined by the mediator \citep{HeisCons2001}. Using the method of improving directions and starting from one predefined solution,  negotiators report their most preferred improving directions at that solution. The mediator decides upon a compromise direction and chooses a new starting solution along it \citep{EhtaSear2001}. Other methods are based on weighted sum scalarization and subgradient optimization \citep{LouWAppr2016} or  Lagrangian duality \citep{HeisDece1999}.
As argued by \citet{RothSome1985}, in multiparty negotiations the likelihood of reaching an AiO optimal solution is small, and if such a solution cannot be achieved, a suboptimal solution is sought \citep{TeicIden1996}. 

The reader is referred to \citet{EngaInte2010} for a comprehensive review of methods for decomposing the complex system and coordinating its subsystems to construct the AiO solution from the computed solutions of the subsystems.

In the works we reviewed above, the authors seek to find an AiO optimal solution for the complex system which, as we already emphasized, may not be possible in general. We are aware of two disciplines that recognize this issue and  in which an AiO feasible but not necessarily optimal solution is sought. In consensus optimization, 
the concepts of distributed computing and coordination on a network are combined which allows subsystems (agents) to individually optimize their own objective function while exchanging information directly or indirectly with other subsystems in the network. Through an iterative process, the subsystems attempt to reach a consensus in the form of a feasible solution that is not necessarily optimal \citep{BertPara1997}. In group decision making, the search for best agreement is not limited to AiO optimal solution but seeks a solution with a high measure of collective satisfaction \citep{FernAnou2013}.


\paragraph{Contribution and Content}
The mathematical models of complex systems and algorithms for decision making  have  provided neither a general model of a complex multiobjective system nor methods generating solutions that are relevant  if an AiO optimal solution is not available.     
In this paper, we extend the state-of-the-art theory of multiobjective optimization to model a  system whose complexity is reflected in the interaction among the subsystems, a feature that has not been addressed before in a multiobjective setting. The complexity requires that domination-based efficiency, which recognizes conflict between the objective functions, be lifted to the new concept  of system-domination-based superiority, which accounts for conflict between subsystems as well as their objective functions. 
The goal of decision making is to find superior solutions for the complex system without ever dealing with this system in its entirety.  The complex systems for which this task is easily  achieved are identified. When superior solutions are unattainable, methods for finding compromise or consensus solutions are proposed.

The paper is structured as follows. In Section~\ref{sec:def} foundations of the new theory are given and a new concept of efficiency, called superiority, is introduced. In Section~\ref{sec:achsuperiority} we develop methods and algorithms for computing superior solutions. Finally, Section~\ref{sec:compromise} deals with the definition and computation of compromise solutions. Section~\ref{sec:conclusion} summarizes the results.

\section{Foundations for a Theory of Multiobjective Complex Systems}\label{sec:def}
\noindent We first  briefly review the basic concepts in multiobjective optimization and build on them a theory of multiobjective complex systems. We define the complex system by means of a graph that implies a decomposition of the system into subsystems. The common notion of feasibility is extended to capture the conceptual difference between the subsystem feasiblity and  new requirements modeling the interaction among the  subsystems. We propose the new concept of superiority to recognize the performance of subsystems within the system and relate it to the classical concept of efficiency.

For comparing vectors $\u,\v\in\R^p$, we use the relations $\leqq, \leq$ and $<$: We write $\u\leqq \v$ if $u_i\leq v_i$ for all $i=1,\dots,p$, $\u\leqslant \v$ if $\u\leqq \v$ and $\u\neq \v$, and $\u<\v$ if $u_i<v_i$ for all $i=1,\dots,p$. 
Consider a multiobjective optimization problem 
\begin{equation*}\tag{$P$}\label{eq:P}
\begin{array}{ll}
\vmin & \f(\x) \\
\text{s.\,t.} & \x\in X
\end{array}
\end{equation*}    
with \(\f\colon \mathbb{R}^{n}\longrightarrow \mathbb{R}^p, X\subseteq \mathbb{R}^{n}\). We assume that \eqref{eq:P}  is a \emph{virtual} optimization problem containing all aspects of the complex system at hand and, therefore, we refer to this problem as the all-in-one problem (AiO problem). This means, that \eqref{eq:P} models a real-life decision making situation and the mathematical model is available but, due to the reasons discussed in Section~\ref{sec:intro}, the problem is never solved. Nevertheless, since the AiO problem is useful for developing a theory of multiobjective complex systems,  we recap notions of optimality in multiobjective optimization.

The symbol \emph{vmin} in \eqref{eq:P} denotes the operator of minimizing the objective functions according to the classical concept of efficiency or Pareto optimality. A feasible solution $\x\in X$ is called \emph{(weakly) efficient}  to problem \eqref{eq:P} (or \emph{AiO-(weakly) efficient}) if there does not exist $\bar{\x}\in X$ such that $\f(\bar{\x}) ( < ) \leqslant \f(\x)$. 
A feasible solution $\x\in X$ is called \emph{strictly efficient} to problem \eqref{eq:P}(or \emph{AiO-strictly efficient}) if there does not exist $\bar{\x}\in X\setminus\{\x\}$ such that $\f(\bar{\x}) \leqq \f(\x)$. 
The sets of efficient solutions, weakly efficient solutions, and strictly efficient solutions to \eqref{eq:P} are denoted by $E(P)$, $wE(P)$, and $sE(P)$, respectively. 

The images $\f(\x)$ of feasible solutions in the objective space are referred to as outcomes, 
their union constitutes the outcome set $\f(X)$. With $(\f(wE(P)))$ and $\f(E(P))$ we denote the sets of \textit{weakly nondominated outcomes} and \textit{nondominated outcomes}, respectively, to problem $P$.


\subsection{A Graph-Based Model and the Concept of Superiority}

\noindent For the AiO problem \eqref{eq:P}, we represent the complex system by a graph model which is associated with  \eqref{eq:P}. This graph carries information of two types: (i) the composition of the complex system consisting of subsystems and (ii) the  assignment of the variables (elements of the vector $\x$) and linking constraints to subsystems. The variables and linking constraints model the interaction among subsystems.

\begin{definition}[Complex system graph]\label{def:complexsystem}
	A \emph{complex system graph} is defined as a directed graph \(G=(\bm{\mathcal{V}}\cup\bm{\mathcal{S}}\cup\bm{\mathcal{C}}, R(\bm{\mathcal{V}},\bm{\mathcal{S}})\cup R(\bm{\mathcal{S}},\bm{\mathcal{C}}))\), where
	\begin{enumerate}[label=(\alph*)]
		\item \(\bm{\mathcal{V}}=\{{\xi}_1,\ldots,{\xi}_{|\bm{\mathcal{V}}|}\}\) denotes the set of \emph{variable nodes} associated with variables \(x_1,\ldots,x_{|\bm{\mathcal{V}}|}\) and the set of indices $\{1,\ldots,|\bm{\mathcal{V}}|\}$ of these nodes.
		\item \(\bm{\mathcal{S}}=\{\sigma_1,\ldots,\sigma_{|\bm{\mathcal{S}}|}\}\) denotes the set of \emph{subsystem nodes} associated with subsystems \(s_1,\ldots,s_{|\bm{\mathcal{S}}|}\) and the set of indices $\{1,\ldots,|\bm{\mathcal{S}}|\}$ of these nodes. 
		\item \(\bm{\mathcal{C}}=\{\kappa_1,\ldots,\kappa_{|\bm{\mathcal{C}}|}\}\) denotes the set of 
		\emph{linking nodes} associated with linking constraints \(c_1,\ldots,c_{|\bm{\mathcal{C}}|}\) and the set of indices $\{1,\ldots,|\bm{\mathcal{C}}|\}$ of these nodes.
		\item \(R(\bm{\mathcal{V}},\bm{\mathcal{S}})\subset\bm{\mathcal{V}}\times\bm{\mathcal{S}}\) denotes the set of  arcs from the nodes in \(\bm{\mathcal{V}}\) to the nodes in \(\bm{\mathcal{S}}\) and \(R(\bm{\mathcal{S}},\bm{\mathcal{C}})\subset\bm{\mathcal{S}}\times\bm{\mathcal{C}}\) denotes the set of arcs from the nodes 
		in \(\bm{\mathcal{S}}\) to the nodes in \(\bm{\mathcal{C}}\).
	\end{enumerate}
\end{definition}

Note that $G$ is a bipartite graph with independent sets $\bm{\mathcal{S}}$ and $\bm{\mathcal{V}}\cup\bm{\mathcal{C}}$.
For notational convenience, subsets of $\bm{\mathcal{V}}$, $\bm{\mathcal{S}}$, or $\bm{\mathcal{C}}$ will also denote subsets of variable nodes, subsystem nodes, or linking nodes accordingly, or the corresponding indices, which will be clear from the context. 
We refer to subsets of subsystems by $S\subseteq\Sys$, to subsets of linking constraints by $C\subseteq\bm{\mathcal{C}}$, and to subsets of objective functions by $F\subseteq\Sys$. Note that $F\subseteq S$ is most meaningful, however, we do not require this in general.

A complex system graph $G$ can be associated with an optimization problem \eqref{eq:P} to represent a feasible decomposition of this problem into a finite set of subsystems. The arcs between variable nodes and subsystem nodes indicate those subsets of variables which are relevant for the respective subsystem. The arcs between subsystem nodes and linking nodes indicate the interconnections between different subsystems by linking constraints which are spanning over several subsystems. 

\begin{figure}
	\centering
	\begin{tikzpicture}[scale=0.8]
	\node[] (Variables) at (0,10.3){Variable nodes $\bm{\mathcal{V}}$};
	\node[] (Subsystems) at (5,10.3){Subsystem nodes $\bm{\mathcal{S}}$};
	\node[] (Linking) at (10,10.3){Linking nodes $\bm{\mathcal{C}}$};
	\node[] (RVS) at (2.5,3.8){$R(\bm{\mathcal{V,S}})$};
	\node[] (RSC) at (7.5,3.8){$R(\bm{\mathcal{S,C}})$};
	\draw[->,thick] (1,3.8) to (RVS);
	\draw[->,dashed,thick] (6,3.8) to (RSC);
	\node [draw, circle,minimum size=0.8cm,inner sep=2pt] (xilast) at (0,5) {$ \xi_{|\bm{\mathcal{V}}|} $};
	\node[] (xidots) at (0,6.5) {$ \vdots $};
	\node [draw, circle,minimum size=0.8cm,inner sep=2pt] (xi2) at (0,8) {$ \xi_2 $};
	\node [draw, circle,minimum size=0.8cm,inner sep=2pt] (xi1) at (0,9.5) {$ \xi_1 $};
	\node[draw, rectangle,minimum size=0.8cm] (S1) at (5,9) {$ \sigma_1 $};
	\node[draw, rectangle,minimum size=0.8cm] (S2) at (5,7.5) {$ \sigma_2 $};
	\node[] (Sdots) at (5,6.1) {$ \vdots $};
	\node[draw, rectangle,minimum size=0.8cm] (Slast) at (5,5) {$ \sigma_{|\bm{\mathcal{S}}|}  $};
	\node[draw, diamond,minimum size=1cm,inner sep=1pt] (Link1) at (10,9) {$ \kappa_1 $};
	\node[] (Link2) at (10,7) {$ \vdots $};
	\node[draw, diamond,minimum size=1cm,inner sep=1pt] (Link3) at (10,5.3) {$ \kappa_{|\bm{\mathcal{C}}|} $};
	
	\draw[->,thick] (xi1)--(S1);
	\draw [->,thick] (xi1)--(Sdots.north west);
	\draw [->,thick] (xi1)--(S2);
	\draw [->,thick] (xi2)--(S1);
	\draw [->,thick] (xi2)--(Sdots.north west);
	\draw [->,thick] (xi2)--(Sdots.west);
	\draw [->,thick] (xidots.north east)--(S1);
	\draw [->,thick] (xidots.east)--(S2);
	\draw [->,thick] (xidots.east)--(Sdots.west);
	\draw [->,thick] (xidots.south east)--(Sdots.west);
	\draw [->,thick] (xidots.south east)--(Slast.west);
	\draw [->,thick] (xidots.east)--(Slast.west);
	\draw [->,thick] (xilast)--(Sdots.west);
	\draw [->,thick] (xilast)--(Sdots.south west);
	\draw [->,dashed,thick] (S1)--(Link1);
	\draw [->,dashed,thick] (S1)--(Link2.north west);
	\draw [->,dashed,thick] (S2)--(Link1);
	\draw [->,dashed,thick] (S2)--(Link2.west);
	\draw [->,dashed,thick] (S2)--(Link3);
	\draw [->,dashed,thick] (Sdots)--(Link2.west);
	\draw [->,dashed,thick] (Slast)--(Link2.south west);
	\draw [->,dashed,thick] (Slast)--(Link3);
	\begin{scope}[on background layer]
	\node[fill=black!15,fit=(Variables) (xilast),rounded corners] {}; 
	\node[fill=black!15,fit=(Subsystems) (Slast) ,rounded corners] {}; 
	\node[fill=black!15,fit=(Linking) (Link3) ,rounded corners] {};
	\end{scope}
	\end{tikzpicture}
	\caption{\label{fig:complexsystemgraph} Illustration of a complex system graph \(G\) with node set $\bm{\mathcal{V}}\cup\bm{\mathcal{S}}\cup\bm{\mathcal{C}}$ and arc set $R(\bm{\mathcal{V}},\bm{\mathcal{S}})\cup R(\bm{\mathcal{S}},\bm{\mathcal{C}})$.}
\end{figure}
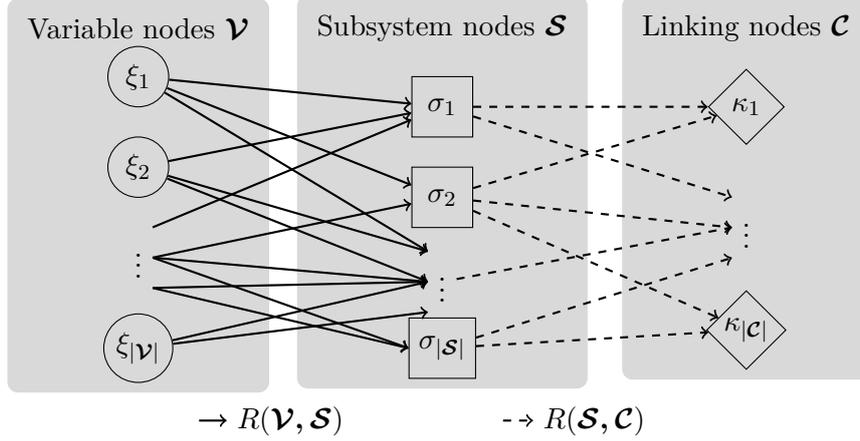


\begin{definition}[Decomposition into subsystems]\label{def:subsystems}
	Consider an optimization problem \eqref{eq:P} and a complex system graph $G$ such that $|\bm{\mathcal{V}}|=n$.
	\begin{enumerate}[label=(\alph*)]
		\item For a vector \(\x=(x_1,\ldots,x_{|\bm{\mathcal{V}}|})^T\in \mathbb{R}^{|\bm{\mathcal{V}}|}\) define a subvector 
		\(\x_{s_i}\in\mathbb{R}^{|\pred(\sigma_i)|}\) of \(\x\) such that 
		\[x_k  \text{ is a component of }\x_{s_i} \Longleftrightarrow \xi_k\in \pred(\sigma_i),\] 
		and define a subvector \(\x_{c_j}\in\mathbb{R}^{|\pred(\pred(\kappa_j))|} \) of \(\x\) such that
		\[x_k \text{ is a component of }\x_{c_j} \Longleftrightarrow \xi_k\in \pred(\pred(\kappa_j)).\]
		%
		
		\item The linking constraint \(c_j\) is of the form 
		\[ c_j(\x_{c_j}) \leqoeq 0, \quad \x_{c_j}\in\mathbb{R}^{|\pred(\pred(\kappa_j))|}.  \]
		To simplify notation, set  $c_j(\x):=c_j(\x_{c_j})$ for $\x\in \mathbb{R}^{|\bm{\mathcal{V}}|}$.
		
		\item The \emph{subsystem \(s_i\)} assumes the form of a multiobjective optimization problem \eqref{eq:subsystemi}:
		\begin{equation*}\tag{$P_{i}$}\label{eq:subsystemi}
		\begin{array}{ll}
		\vmin & \f_i(\x_{s_i}) \\
		\text{s.t.} & \x_{s_i} \in X_i \subset \mathbb{R}^{|\pred(\sigma_i)|},
		\end{array}
		\end{equation*}
		where \(\f_i\colon \mathbb{R}^{|\pred(\sigma_i)|} \longrightarrow \mathbb{R}^{p_i}\) is a vector-valued objective function associated with the subsystem $s_i\in\Sys$, and  $\sum_{i=1}^{|\bm{\mathcal{S}}|} p_i=p$. Its feasible set \(X_i\) is defined by local constraints. 
		
		\item Given all subsystems $s_i$, $i=1, \dots , |\bm{\mathcal{S} }|$,
		the \emph{AiO objective function vector} is composed of the objective functions of the subsystems, i.e.,
		\begin{equation*}
		\f(\x)=(\f_1(\x_{s_1}),\ldots,\f_{|\bm{\mathcal{S}}|}(\x_{s_{|\bm{\mathcal{S}}|}})).
		\end{equation*}
		To simplify notation, set  $\f_i(\x):=\f_i(\x_{s_i})$ for $\x\in \mathbb{R}^{|\bm{\mathcal{V}}|}$ and write \(\x\in X_j\) if and only if \(\x_{s_j}\in X_j\).
		
		\item The complex system graph \(G\) is called a \emph{decomposition for the AiO problem \eqref{eq:P}} provided
		\[\x\in X \Longleftrightarrow \x_{s_i}\in X_i \ \forall i=1,\ldots,|\bm{\mathcal{S}}| \land c_j(\x_{c_j}) \leqoeq 0 \ \forall j=1,\ldots,|\bm{\mathcal{C}}|,\] 
		where $X_i\subseteq\R^{|\pred(\sigma_i)|}$ is the individual feasible set for the subsystem $s_i\in\Sys$.
	\end{enumerate}
\end{definition}

\begin{remark}
	\begin{enumerate}[label=(\alph*)] 
		\item     A variable \(x_k\) is called a \emph{local variable} provided \(|\{\sigma_i \in \bm{\mathcal{S}}: \xi_k \in \pred(\sigma_i)\}| = 1\). A variable \(x_k\)  is called a \emph{global variable} provided \(|\{\sigma_i \in \bm{\mathcal{S}}: \xi_k \in \pred(\sigma_i)\}| > 1\).
		\item Note that a linking constraint $c_j$ only depends on variables associated with the set $\pred(\pred(\kappa_j))\subset\bm{\mathcal{V}}$ where $\pred(w)$  denotes the set of predecessors of node $w$ in $G$.
		\item Since the subsystems $s_i$, $i=1, \dots , |\bm{\mathcal{S} }|$, are MOPs themselves, the common notion of efficiency applies.
	\end{enumerate}
\end{remark}

A decomposition of the AiO problem \eqref{eq:P} based on a complex system graph~$G$  leads to the problem \eqref{eq:PG} which is equivalent to \eqref{eq:P} but, additionally, reveals some structure of the complex system: 
\begin{equation*}\tag{$P(G)$}\label{eq:PG}
\begin{array}{lll}
\smin & \multicolumn{2}{l}{(\f_1(\x_{s_1}),\ldots,\f_{|\bm{\mathcal{S}}|}(\x_{s_{|\bm{\mathcal{S}}|}}))} \\[0.1cm]
\text{s.t.} & \x_{s_i} \in X_i, & i=1,\ldots,|\bm{\mathcal{S}}|\\[0.2cm]
& c_j(\x_{c_j}) \leqoeq 0, &  j=1,\ldots,|\bm{\mathcal{C}}|.
\end{array}
\end{equation*}    



The symbol \emph{smin} in \ref{eq:PG} denotes the operator of minimizing the subsystem objective functions according to a new concept that we now introduce. Complex AiO problems often suggest naturally a decomposition into subsystems or come as a collection of (sub)problems that are interconnected by variables or linking constraints. In general, a decomposition of the AiO problem into subsystems is not unique, and different decompositions may give rise to different solution concepts and methods for the AiO problem. 

We distinguish between individual subsystem feasibility and the interaction among the subsystems which is modeled with linking constraints. The former is referred to as feasibility while the latter is referred to as consistency. In this way, we extend the classical meaning of feasibility and define the term \emph{validity} to represent solutions that are both feasible \emph{and} consistent.

\begin{definition}[Subsystem feasibility, consistency, and validity] \label{def:subsystemandlinkingfeasible}
	Consider an AiO problem~\eqref{eq:P} and its decomposition~\eqref{eq:PG} w.r.t.~an associated complex system graph~$G$.
	Let $S\subseteq\bm{\mathcal{S}}$ be a subset of subsystems and $C\subseteq\bm{\mathcal{C}}$ be a subset of linking constraints. A vector \(\x\in\mathbb{R}^{|\bm{\mathcal{V}|}}\) is called a:
	\begin{enumerate}[label=(\alph*)]
		\item \emph{feasible solution w.r.t.\ the subsystems} in \(S\) 
		(or \emph{$S$-feasible}) if it is feasible to all subsystems contained in $S$, i.\,e., $\x_{s_i}\in X_i$ for all $i\in S$. 
		The set of \(S\)-feasible solutions is denoted by \(X_{S,\emptyset}\subseteq\mathbb{R}^{|\bm{\mathcal{V}|}}\).
		\item \emph{consistent solution w.r.t.\ the linking constraints in \(C\)} 
		(or \emph{$C$-consistent}) if \( c_j(\x_{c_j}) \leqoeq 0\) holds for all $j\in C$.
		The set of \(C\)-consistent solutions is denoted by \(X_{\emptyset, C}\subseteq\mathbb{R}^{|\bm{\mathcal{V}|}}\).
		\item \emph{($S,C$)-valid solution} if it is $S$-feasible and $C$-consistent.
		The set of all $(S,C)$-valid solutions is denoted by $X_{S,C}\subseteq\mathbb{R}^{|\bm{\mathcal{V}|}}$.
		\item \emph{system feasible solution} if it is $\bm{\mathcal{S}}$-feasible.  It is called a 
		\emph{system consistent solution} if it is $\bm{\mathcal{C}}$-consistent. It is called a \emph{system valid solution} if it is $(\bm{\mathcal{S}},\bm{\mathcal{C}})$-valid.
		The all system feasible set is denoted by $X_{\bm{\mathcal{S}},\emptyset}$, the all system consistent set is denoted by $X_{\emptyset,\bm{\mathcal{C}}}$, and the system valid set is denoted by $X_{\bm{\mathcal{S}},\bm{\mathcal{C}}}=X$.
	\end{enumerate}
	
\end{definition}

To simplify notation in the case that $|S|=1$ or $|C|=1$, i.\,e., $S=\{i\}$ or $C=\{j\}$, we shortly write $X_{i,C}$, $X_{S, j}$, or $X_{i,j}$, respectively. 
Moreover, in the case that a constraint set, for example, on the level of subsystem $s_i$, is extended by additional equations or by the intersection with another constraint set $Y$,  we may alternatively write $X_{X_i\cap Y,C}$ 
and refer to the corresponding solutions as being $(X_i\cap Y,C)$-valid, 
slightly abusing the notation.
Notice that in general \(X_{i,\emptyset}=\{\x\in\mathbb{R}^{|\bm{\mathcal{V}|}}\colon\x_{s_i}\in X_i\}\neq X_i\subseteq \mathbb{R}^{|\pred(\sigma_i)|}\).


\subsection{Linking Constraints versus Global Variables}
\noindent In the following, we will discuss the interchangeable role of linking constraints and global variables. 
First we will describe how linking constraints can be reformulated by means of global variables, and later how global variables can be substituted by a specific type of linking constraints, which will be called \emph{easy-linking}.

A linking constraint $c_k(\x_{c_j})\leqoeq 0$, linking a set $S_k\subseteq\bm{\mathcal{S}}$ of subsystems ($|S_k|>1$), can be substituted by a local constraint that is added in each subsystem of $S_k$. If we do so, all local variables of all subsystems \(s_i\in S_k\) 
occurring in this added constraint $c_k(\x_{c_j})\leqoeq 0$ are now present in all subsystems in $S_k$ and thus become global variables.

On the other hand, a global variable \(x_k\) can be substituted by a local copy $x_k^{(i)}$ in each subsystem $s_i$  in which it is involved. 
These local copies are equated by a set of \emph{easy-linking} constraints of the form \(x_k^{(i)}=x_k^{(j)}\), ensuring pairwise equality for all subsystems $s_i$ and $s_j$ using the respective variable \(x_k\) (i.e.\ \(\xi_k\in\pred(\sigma_i)\cap \pred(\sigma_j) \)).
Note that pairwise easy-linking constraints ensure the equality of all local copies for all system valid solutions. 

Applying both reformulations we can adapt the graph representation of the complex system by removing the linking nodes and establishing an arborescence (minimum spanning digraph) between the nodes corresponding to local copies of one global variable. We call these arcs between local copies of a global variable \emph{easy-linking arcs}. Let \(A\) be the incidence matrix of all these easy-linking arcs, then the set of easy-linking constraints can be written as \(A\, \x = \bm0\). Then, the number of newly introduced copies of variable nodes equals $|R(\bm{\mathcal{V,S}})|-|\bm{\mathcal{V}}|$, resulting in $|R(\bm{\mathcal{V,S}})|$ variable nodes all corresponding to local variables. 

While the substitution of linking constraints by global variables does not change the feasible set of the AiO problem \eqref{eq:P}, the introduction of local copies lifts the set of system valid solutions to a higher dimension. In addition, since restrictions are interchangeably related to feasibility or to consistency, subsystem validity (and thus optimality) can be affected by both reformulations.
This has consequences as well on relaxations, bounds, iterative solution schemes and optimality concepts as discussed in Section \ref{subsec:super} below.
Even more so, the interpretation of the complex system is changed in many ways. For example, a linking constraint is usually interpreted as an interface between two departments negotiating about the solution. If this linking constraint is now moved into one of the subsystems, then the corresponding department takes the responsibility for a valid cooperation for both departments by possibly adjusting preiviously local variables of the competing department.

If all linking constraints are substituted by global variables and then all global variables are replaced by easy-linking, we obtain a complex system that contains only local variables, local constraints and easy-linking constraints. This motivates the notion of a complex system in standard form:
\begin{definition}[Standard Form]
	An optimization problem 
	\begin{equation*}\tag{$P(G)_{\text{sf}}$}\label{eq:PGSF}
	\begin{array}{lll}
	\smin & \multicolumn{2}{l}{(\f_1(\x_{s_1}),\ldots,\f_{|\bm{\mathcal{S}}|}(\x_{s_{|\bm{\mathcal{S}}|}}))} \\[0.1cm]
	\text{s.t.} & \x_{s_i} \in X_i, & i=1,\ldots,|\bm{\mathcal{S}}|\\[0.2cm]
	& A \, \x=\bm0
	\end{array}
	\end{equation*}  
	is called a \emph{complex system in standard form}.
\end{definition}

\begin{example}\label{ex:globaltoeasylinking}
	Consider the graph representation of a problem with four variables, three subsystems and one linking constraint illustrated in the left upper part of Figure~\ref{fig:exampleglobaltoeasylinking}. 
	The reformulation of linking constraints in terms of local constraints is illustrated in the left lower part of Figure~\ref{fig:exampleglobaltoeasylinking}. Thus, variable node \(\xi_1\) is also adjacent to system node \(\sigma_3\) (and \(\xi_3\) is adjacent to \(\sigma_1\)).
	The right part of Figure \ref{fig:exampleglobaltoeasylinking} shows the graph representation after the transformation into standard form, where the dashed arcs denote easy-linking constraints between two variables. 
	For the new variable vector $\x=\bigl(x_1^{(1)},x_1^{(2)},x_1^{(3)},x_2,x_3,x_4^{(1)},x_4^{(3)}\bigr)^T$, the incidence matrix representing the easy-linking arcs is given by
	$$ A= \begin{pmatrix}
	1 & -1 & 0 & 0 & 0 & 0 & 0 \\
	0 & 1 & -1 & 0 & 0 & 0 & 0 \\
	0 & 0 & 0  & 0 & 0 & 1 & -1 	
	\end{pmatrix} .$$
	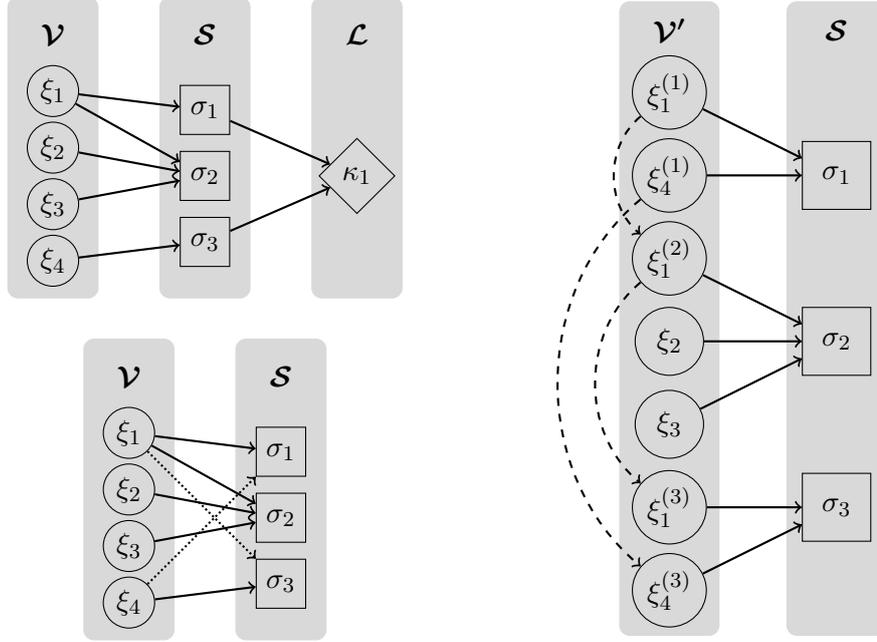
\begin{figure}[htb]
		\hfill
		\begin{minipage}{0.4\textwidth}
	\centering
	\begin{tikzpicture}[scale=0.5]
	\node[] (Variables) at (0,6){
		$\bm{\mathcal{V}}$};
	\node[] (Subsystems) at (4,6){
		$\bm{\mathcal{S}}$};
	\node[] (Linking) at (8,6){
		$\bm{\mathcal{L}}$};
	
	\node [draw, circle,minimum size=0.65cm,inner sep=2pt] (xi1) at (0,4.5) {$ \xi_1 $};
	\node [draw, circle,minimum size=0.65cm,inner sep=2pt] (xi2) at (0,3) {$ \xi_2 $};
	\node [draw, circle,minimum size=0.65cm,inner sep=2pt] (xi3) at (0,1.5) {$ \xi_3 $};
	\node [draw, circle,minimum size=0.65cm,inner sep=2pt] (xi4) at (0,0) {$ \xi_4 $};
	\node[draw, rectangle,minimum size=0.65cm,inner sep=2pt] (S1) at (4,4) {$ \sigma_1 $};
	\node[draw, rectangle,minimum size=0.65cm,inner sep=2pt] (S2) at (4,2.25) {$ \sigma_2 $};
	\node[draw, rectangle,minimum size=0.65cm,inner sep=2pt] (S3) at (4,0.5) {$ \sigma_3 $};
	\node[draw, diamond,minimum size=1cm,inner sep=1pt] (L1) at (8,2.25) {$ \kappa_1 $};
	
	\draw[->,thick] (xi1)--(S1);
	\draw[->,thick] (xi1)--(S2);
	\draw[->,thick] (xi2)--(S2);
	\draw[->,thick] (xi3)--(S2);
	\draw[->,thick] (xi4)--(S3);
	\draw [->,thick] (S1)--(L1);
	\draw [->,thick] (S3)--(L1);
	\begin{scope}[on background layer]
	\fill [fill=black!15,rounded corners] (-1.2,-1) rectangle (1.2,7);
	\fill [fill=black!15,rounded corners] (2.8,-1) rectangle (5.2,7);
	\fill [fill=black!15,rounded corners] (6.8,-1) rectangle (9.2,7);
	\end{scope}
	\end{tikzpicture}
	\\[3ex]
	\centering
	\begin{tikzpicture}[scale=0.5]
	\node[] (Variables) at (0,6){
		$\bm{\mathcal{V}}$};
	\node[] (Subsystems) at (4,6){
		$\bm{\mathcal{S}}$};
	
	\node [draw, circle,minimum size=0.65cm,inner sep=2pt] (xi1) at (0,4.5) {$ \xi_1 $};
	\node [draw, circle,minimum size=0.65cm,inner sep=2pt] (xi2) at (0,3) {$ \xi_2 $};
	\node [draw, circle,minimum size=0.65cm,inner sep=2pt] (xi3) at (0,1.5) {$ \xi_3 $};
	\node [draw, circle,minimum size=0.65cm,inner sep=2pt] (xi4) at (0,0) {$ \xi_4 $};
	\node[draw, rectangle,minimum size=0.65cm,inner sep=2pt] (S1) at (4,4) {$ \sigma_1 $};
	\node[draw, rectangle,minimum size=0.65cm,inner sep=2pt] (S2) at (4,2.25) {$ \sigma_2 $};
	\node[draw, rectangle,minimum size=0.65cm,inner sep=2pt] (S3) at (4,0.5) {$ \sigma_3 $};
	
	\draw[->,thick] (xi1)--(S1);
	\draw[->,thick,  densely dotted] (xi1)--(S3);
	\draw[->,thick] (xi1)--(S2);
	\draw[->,thick] (xi2)--(S2);
	\draw[->,thick] (xi3)--(S2);
	\draw[->,thick] (xi4)--(S3);
	\draw[->,thick, densely dotted] (xi4)--(S1);
	\begin{scope}[on background layer]
	\fill [fill=black!15,rounded corners] (-1.2,-1) rectangle (1.2,7);
	\fill [fill=black!15,rounded corners] (2.8,-1) rectangle (5.2,7);
	\end{scope}
	\end{tikzpicture}
\end{minipage}
\hfill
\begin{minipage}{0.45\textwidth}
	\begin{tikzpicture}[scale=0.55]
	\node[] (Variables) at (0,13.5){
		$\bm{\mathcal{V'}}$};
	\node[] (Subsystems) at (4,13.5){
		$\bm{\mathcal{S}}$};
	
	\node [draw, circle,minimum size=0.9cm,inner sep=2pt] (xi11) at (0,12) {$ \xi_1^{(1)} $};
	\node [draw, circle,minimum size=0.9cm,inner sep=2pt] (xi41) at (0,10) {$ \xi_4^{(1)} $};
	\node [draw, circle,minimum size=0.9cm,inner sep=2pt] (xi12) at (0,8) {$ \xi_1^{(2)} $};
	\node [draw, circle,minimum size=0.9cm,inner sep=2pt] (xi2) at (0,6) {$ \xi_2 $};
	\node [draw, circle,minimum size=0.9cm,inner sep=2pt] (xi3) at (0,4) {$ \xi_3 $};
	\node [draw, circle,minimum size=0.9cm,inner sep=2pt] (xi13) at (0,2) {$ \xi_1^{(3)} $};
	\node [draw, circle,minimum size=0.9cm,inner sep=2pt] (xi43) at (0,0) {$ \xi_4^{(3)} $};
	\node[draw, rectangle,minimum size=0.9cm,inner sep=2pt] (S1) at (4,10) {$ \sigma_1 $};
	\node[draw, rectangle,minimum size=0.9cm,inner sep=2pt] (S2) at (4,6) {$ \sigma_2 $};
	\node[draw, rectangle,minimum size=0.9cm,inner sep=2pt] (S3) at (4,2) {$ \sigma_3 $};
	
	\draw[->,thick] (xi11)--(S1);
	\draw[->,thick] (xi41)--(S1);
	\draw[->,thick] (xi12)--(S2);
	\draw[->,thick] (xi2)--(S2);
	\draw[->,thick] (xi3)--(S2);
	\draw[->,thick] (xi13)--(S3);
	\draw[->,thick] (xi43)--(S3);
	\draw[->,thick,dashed] (xi11) to [out=220,in=140] (xi12);
	\draw[->,thick,dashed] (xi12) to [out=220,in=140] (xi13);
	\draw[->,thick,dashed] (xi41) to [out=220,in=140] (xi43);
	\begin{scope}[on background layer]
	\fill [fill=black!15,rounded corners] (-1.2,-1.1) rectangle (1.2,14.2);
	\fill [fill=black!15,rounded corners] (2.8,-1.1) rectangle (5.2,14.2);
	\end{scope}
	\end{tikzpicture}
\end{minipage}
		\caption{\label{fig:exampleglobaltoeasylinking}A complex system graph before (left top), after the rewriting of linking (left bottom) and after the transformation into standard form (right). Note that all variables are local variables partially linked by easy-linking.}
	\end{figure}
\end{example}



%

\subsection{Superiority and Efficiency} \label{subsec:super}


\noindent Since the performance of feasible solutions to the complex system shall be studied with respect to the subsystems and their objective functions, we introduce the concept of \emph{system dominance}. These objective functions are specified by the set  \(F\subseteq\bm{\mathcal{S}}\) of the indices of the subsystems they belong to.

\begin{definition}[(Strict/weak) system dominance]\label{def:systemdominance}
	Consider an optimization problem~\eqref{eq:P} and its decomposition~\eqref{eq:PG} w.r.t.\ an associated complex system graph~$G$.
	Let \(S\subseteq\bm{\mathcal{S}}\) be a set of subsystems, \(F\subseteq\bm{\mathcal{S}}\) be a set of objective functions,  and \(C\subseteq \bm{\mathcal{C}}\) be a set of linking constraints. A solution $\bar{\x}\in X_{S,C}$ is said to:
	\begin{enumerate}[label=(\alph*)]
		\item  \emph{\((F,S,C)\)-strictly system dominate} $\x\in X_{S,C}$ if 
		\( \forall i\in F: \f_i(\bar{\x})\leqslant \f_i(\x).\)
		
		We write $\f(\bar{\x}) \prec_{(F,S,C)} \f(\x)$ in this case.
		\item  \emph{\((F,S,C)\)-system dominate} $\x\in X_{S,C}$ if 
		\( \forall i\in F : \f_i(\bar{\x}) \leqq \f_i(\x) \land \exists j\in F: \f_j(\bar{\x}) \leqslant \f_j(\x) .\)
		
		We write $\f(\bar{\x}) \preceq_{(F,S,C)} \f(\x)$ in this case.
		\item \emph{\((F,S,C)\)-weakly system dominate} $\x\in X_{S,C}$ if
		\(\forall i\in F : \f_i(\bar{\x}) \leqq \f_i(\x) .\)
		
		We write $\f(\bar{\x}) \precqq_{(F,S,C)} \f(\x)$ in this case.
	\end{enumerate}
\end{definition}

\begin{remark}
	Let \(F,S\subseteq\bm{\mathcal{S}}\) and \(C\subseteq \bm{\mathcal{C}}\), and consider two solutions \(\x,\bar{\x}\in X_{S,C}\). Then $\f(\bar{\x}) \prec_{(F,S,C)} \f(\x)$ implies 
	$\f(\bar{\x}) \preceq_{(F,S,C)} \f(\x)$,
	which implies $\f(\bar{\x}) \precqq_{(F,S,C)} \f(\x)$.
\end{remark}

Note that in the case that $|F|=1$, both (strict) system dominance correspond to the classical concept of dominance, while weak system dominance corresponds to the classical concept of weak dominance. 

Using the notion of system dominance, we define a new concept of optimality for complex systems making use of a new  term \emph{superior} which replaces the term \emph{efficient} for MOPs.

\begin{definition}[Superior, weakly superior, and strictly superior solutions]\label{def:efficient}
	Consider an optimization problem~\eqref{eq:P} and its decomposition~\eqref{eq:PG} w.r.t.~an associated complex system graph~$G$.
	Let  \(S\subseteq\bm{\mathcal{S}}\) denote a set of subsystems, \(F\subseteq\bm{\mathcal{S}}\) denote a set of objective functions, and \(C\subseteq \bm{\mathcal{C}}\) denote a set of linking constraints. An \((S,C)\)-valid  solution \(\x\in X_{S,C}\) is called:
	\begin{enumerate}[label=(\alph*)]
		\item \emph{\((F,S,C)\)-weakly superior} if 
		\(\nexists \bar{\x}\in X_{S,C}: \f(\bar{\x})\prec_{(F,S,C)} \f(\x).\)
		
		The set of all $(F,S,C)$-weakly superior solutions is denoted by \(\wSup(F,S,C)\).
		\item \emph{\((F,S,C)\)-superior} if 
		\(\nexists \bar{\x}\in X_{S,C}: \f(\bar{\x})\preceq_{(F,S,C)} \f(\x) .\)
		
		The set of all $(F,S,C)$-superior solutions is denoted by \(\Sup(F,S,C)\).
		\item \emph{\((F,S,C)\)-strictly superior} if 
		\(\nexists \bar{\x}\in X_{S,C}\setminus\lbrace \x\rbrace:  \f(\bar{\x}) \precqq_{(F,S,C)} \f(\x).\)
		 
		The set of all $(F,S,C)$-strictly superior solutions is denoted by \(\sSup(F,S,C)\).
	\end{enumerate}
\end{definition}

To simplify notation in the case that $|F|=1$, $|S|=1$ or $|C|=1$, i.\,e., $F=\{i\}$, $S=\{j\}$ or $C=\{k\}$, we shortly write $\Sup(i,S,C)$, $\Sup(F,j,C)$, or $\Sup(F,S,k)$, respectively. The same applies to the notation of weakly and strictly superior sets.

\begin{remark} Under the assumptions of Definition \ref{def:efficient}, the following holds:
	\begin{enumerate}
		\item In the case that $|F|=1$, both concepts of weakly superior and superior solutions reduce to the classical concept of efficient solutions, while the concept of strict superior solutions reduces to the classical concept of strictly efficient solutions.
		\item For every subsystem $s_i$, the problem \eqref{eq:subsystemi} is equivalent to the case $F=S=\{i\}$ and $C=\emptyset$, which implies that $E(P_i) = \Sup(i,i,\emptyset)$ and $wE(P_i) = \wSup(i,i,\emptyset)$. 
	\end{enumerate}
\end{remark}


The concept of superiority is conditioned by subsets of objective functions, subsystem constraints, and linking constraints. It thus allows a meaningful comparison of the superior sets for different \emph{coalitions} among the subsystems. In this respect, a coalition may be interpreted as a triple $(F,S,C)$, where usually $F\subseteq S$ and $C$ contains all linking constraints that interrelate the subsystems in $S$. When adding additional objectives and/or constraints, then the respective superior sets are contained within each other as can be seen from the following results.

\begin{proposition}\label{prop:FsupToSsup}
	Let  \(S\subseteq\bm{\mathcal{S}}\) denote a set of subsystems, let \(F_1,F_2\subseteq\bm{\mathcal{S}}\) denote two sets of objective functions such that $\emptyset\neq F_1\subset F_2$, and let \(C\subseteq \bm{\mathcal{C}}\) denote a set of linking constraints. Consider a solution $\x \in X_{S,C}$. If $\x$ is
	\begin{enumerate}[label=(\alph*)]
		\item $(F_{1},S,C)$-weakly superior, then $\x$ is $(F_2,S,C)$-weakly superior, i.e.,\\ $\wSup(F_1,S,C)\subseteq\wSup(F_2,S,C).$
		\item $(F_{1},S,C)$-strictly superior, then $\x$ is $(F_2,S,C)$-stricty superior, i.e.,\\ $\sSup(F_1,S,C)\subseteq\sSup(F_2,S,C).$
	\end{enumerate}
\end{proposition}
\begin{proof}
	\begin{enumerate}[label=(\alph*)]
		\item Suppose that $\x\in X_{S,C}$ and $\x \notin  \wSup(F_2,S,C)$. Then there exists a solution $\bar{\x} \in X_{S,C}$ such that $\f(\bar{\x}) \prec_{(F_2,S,C)} \f(\x)$, i.e., $\f_i(\bar{\x})\leqslant\f_i(\x)$ for all $i\in F_2$. Since $F_1\subset F_2$, this immediately implies that $\f(\bar{\x}) \prec_{(F_1,S,C)} \f(\x)$. 
		Then, $\x\not\in\wSup(F_1,S,C)$.
		%
		%
		\item Suppose that $\x\in X_{S,C}$ and $\x \notin  \sSup(F_2,S,C)$. Then there exists a solution $\bar{\x} \in X_{S,C}\setminus\{\x\}$ such that $\f(\bar{\x}) \precqq_{(F_2,S,C)} \f(\x)$, i.e., $\f_i(\bar{\x})\leqq\f_i(\x)$ for all $i\in F_2$. Since $F_1\subset F_2$, this immediately implies that $\f(\bar{\x}) \precqq_{(F_1,S,C)} \f(\x)$. Then, $\x\not\in\wSup(F_1,S,C)$.\qedhere
	\end{enumerate}
\end{proof}

Note that Proposition \ref{prop:FsupToSsup} implies in particular that for all $F\subseteq\bm{\mathcal{S}}$, $F\neq\emptyset$, we have that
$\wSup(F,\bm{\mathcal{S}},\bm{\mathcal{C}})\subseteq\wSup(\bm{\mathcal{S}},\bm{\mathcal{S}},\bm{\mathcal{C}})$ and $\sSup(F,\bm{\mathcal{S}},\bm{\mathcal{C}})\subseteq\sSup(\bm{\mathcal{S}},\bm{\mathcal{S}},\bm{\mathcal{C}})$.
Moreover, $\Sup(F_1,S,C)\subseteq\wSup(F_2,S,C)$ whenever $\emptyset\neq F_1\subseteq F_2$. However, in general we do \emph{not} have $\Sup(F_1,S,C)\subseteq\Sup(F_2,S,C)$ as can be seen from Example \ref{ex:prop:FsupToSsup} below.

\begin{example}\label{ex:prop:FsupToSsup} Consider a complex system \eqref{eq:P} decomposed into two subsystems $s_1$ and $s_2$ with equal feasible sets $X_1=X_2=\{\x\in\mathbb{R}^2 : 0\leq x_i\leq 1,\, i=1,2\}$, no linking constraints, and objective functions $\f_1(\x)=(x_1)$ and $\f_2(\x)=(x_2)$. Then $\f(X_{\{1,2\},\emptyset})=[0,1]\times[0,1]$, $\Sup(1,\{1,2\},\emptyset)=\{0\}\times[0,1]$ and 
	$\Sup(\{1,2\},\{1,2\},\emptyset)=(\{\0\})$. 
\end{example}

\begin{proposition}\label{prop:subsetsup}
	Let  \(S_1,S_2\subseteq\bm{\mathcal{S}}\) denote two sets of subsystems such that $S_1\subseteq S_2$, and let \(F\subseteq\bm{\mathcal{S}}\) denote a set of objective functions.
	Consider a solution $\x \in X_{S_2,\bm{\mathcal{C}}}$, i.e., \(\x\) is \((S_2,\bm{\mathcal{C}})\)-valid for \((P(G))\). If \(\x\) is 
	\begin{enumerate}[label=(\alph*)]
		\item \((F,S_1,\bm{\mathcal{C}})\)-weakly superior, then \(\x\) is \((F,S_2,\bm{\mathcal{C}})\)-superior, i.e.,\\
		$\wSup(F,S_1,\bm{\mathcal{C}})\subseteq\wSup(F,S_2,\bm{\mathcal{C}}).$
		\item \((F,S_1,\bm{\mathcal{C}})\)-superior, then \(\x\) is \((F,S_2,\bm{\mathcal{C}})\)-superior, i.e.,\\
		$\Sup(F,S_1,\bm{\mathcal{C}})\subseteq\Sup(F,S_2,\bm{\mathcal{C}}).$
		\item \((F,S_1,\bm{\mathcal{C}})\)-strictly superior, then \(\x\) is \((F,S_2,\bm{\mathcal{C}})\)-strictly superior, i.e.,\\
		$\sSup(F,S_1,\bm{\mathcal{C}})\subseteq\sSup(F,S_2,\bm{\mathcal{C}}).$
	\end{enumerate}
\end{proposition}

\begin{proof} We prove part (b); parts (a) and (c) follow analogously.
	Let \(\x\in X_{S_2,\bm{\mathcal{C}}}\) be \((F,S_1,\bm{\mathcal{C}})\)-superior. Then, there does not exist  \(\bar{\x} \in X_{S_1,\bm{\mathcal{C}}}\) with \( \f_i(\bar{\x}) \leqslant \f_i(\x)\) for all $i\in F$. Since \(X_{S_2,\bm{\mathcal{C}}} \subseteq X_{S_1,\bm{\mathcal{C}}}\), this implies that there does not exist \(\bar{\x} \in X_{S_2,\bm{\mathcal{C}}}\) with \( \f_i(\bar{\x}) \leqslant \f_i(\x)\) for all $i\in F$. Therefore, \(\x\) is \((F,S_2,\bm{\mathcal{C}})\)-superior for \((P(G))\).
\end{proof}

Note that Proposition \ref{prop:subsetsup} implies in particular that for all $i\in\bm{\mathcal{S}}$ and for all $S_1\subseteq S_2\subseteq\bm{\mathcal{S}}$ we have that
$\wSup(i,S_1,\bm{\mathcal{C}})\subseteq\wSup(i,S_2,\bm{\mathcal{C}})$,
$\Sup(i,S_1,\bm{\mathcal{C}})\subseteq\Sup(i,S_2,\bm{\mathcal{C}})$, and 
$\sSup(i,S_1,\bm{\mathcal{C}})\subseteq\sSup(i,S_2,\bm{\mathcal{C}})$. The converse of Proposition \ref{prop:subsetsup} is in general not true, 
cf. Example~\ref{example:prop:subsetsup}.
\begin{example}\label{example:prop:subsetsup} 
	Consider a complex system decomposed into the two subsystems $s_1$ with $X_1=\{\x\in\R^2: x_1, x_2 \geq 0\}$ and \(\f_1(\x) = (x_1+x_2)\), and $s_2$ with \(X_2 = \{\x\in\R^2: x_1+x_2 \geq 2\}\) and \(\f_2(\x) = (0)\). Suppose that there are no linking constraints, i.e., $\bm{\mathcal{C}}=\emptyset$.
	The vector \(\x = (1,1)^T\) is \((1,\{1,2\},\emptyset)\)-superior but not \((1,1,\emptyset)\)-superior, since \((0,0)^T\) is \((1,\emptyset)\)-valid and thus $(0,0)^T\prec_{(1,1,\emptyset)}(1,1)^T$. See Figure \ref{fig:prop:subsetsup} for an illustration.
\end{example}

\begin{figure}
	\centering
\begin{tikzpicture}[scale=0.5]
\draw (0,0) -- (0,-0.08);
\draw (1,0) -- (1,-0.08);
\draw (2,0) -- (2,-0.08);
\node [below] at (2,0) {$1$};
\draw (3,0) -- (3,-0.08);
\draw (4,0) -- (4,-0.08);
\node [below] at (4,0) {$2$};
\node [below] at (5,0) {$x_1$};
\draw (0,0) -- (-0.08,0);
\draw (0,1) -- (-0.08,1);
\draw (0,2) -- (-0.08,2);
\node [left] at (0,2) {$1$};
\draw (0,3) -- (-0.08,3);
\draw (0,4) -- (-0.08,4);
\node [left] at (0,4) {$2$};
\node [left] at (0,5) {$x_2$};
\path [fill=lightgray] (0,0) -- (4.8,0) -- (4.8,4.8) -- (0,4.8) -- (0,0);
\node at (3.8,3.8) {$X_{1,\emptyset}$};
\draw [->] (0,0) -- (5,0);
\draw [->] (0,0) -- (0,5);
\draw [fill=white] (2,2) circle (.5ex);
\draw [fill=black] (0,0) circle (.5ex);
\end{tikzpicture}\hspace{3em}
\begin{tikzpicture}[scale=0.5]
\draw (0,0) -- (0,-0.08);
\draw (1,0) -- (1,-0.08);
\draw (2,0) -- (2,-0.08);
\node [below] at (2,0) {$1$};
\draw (3,0) -- (3,-0.08);
\draw (4,0) -- (4,-0.08);
\node [below] at (4,0) {$2$};
\node [below] at (5,0) {$x_1$};
\draw (0,0) -- (-0.08,0);
\draw (0,1) -- (-0.08,1);
\draw (0,2) -- (-0.08,2);
\node [left] at (0,2) {$1$};
\draw (0,3) -- (-0.08,3);
\draw (0,4) -- (-0.08,4);
\node [left] at (0,4) {$2$};
\node [left] at (0,5) {$x_2$};
\path [fill=lightgray] (0,4) -- (4,0) -- (4.8,0) -- (4.8,4.8) -- (0,4.8) -- (0,4) ;
\draw  (4,0) -- (0,4);
\node at (3.5,3.8) {$X_{\{1,2\},\emptyset}$};
\draw [->] (0,0) -- (5,0);
\draw [->] (0,0) -- (0,5);
\draw [fill=white] (2,2) circle (.5ex);
\draw [fill=black] (0,0) circle (.5ex);
\end{tikzpicture}
	\caption{\label{fig:prop:subsetsup} Illustration of Example~\ref{example:prop:subsetsup}. \((1,1)^T\) is marked white, \((0,0)^T\) is marked black.}
\end{figure}
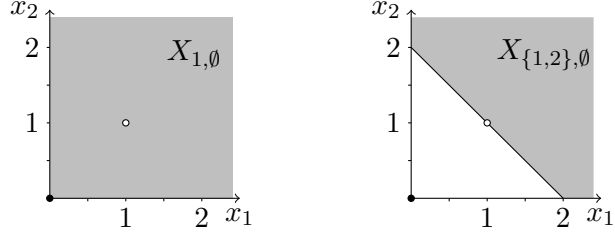

\subsection{Relation to AiO Efficiency}

\noindent In the following, the new concept of superiority for problem \eqref{eq:PG} is examined and contrasted with the classical concept of efficiency for problem \eqref{eq:P}. We first show equivalence of these two concepts in the presence of all subsystems and all their objective functions.

\begin{proposition}\label{lemma:superiority_versus_efficieny}
	Let $\x\in X$, i.e., \(\x\in\mathbb{R}^{|\bm{\mathcal{V}|}}\) is system valid. Then the following holds:
	\begin{enumerate}[label=(\alph*)]
		\item $\x$ is $(\bm{\mathcal{S}},\bm{\mathcal{S}},\bm{\mathcal{C}})$-superior for \eqref{eq:PG} if and only if $\x$ is efficient for \eqref{eq:P}.
		\item If $\x$ is $(\bm{\mathcal{S}},\bm{\mathcal{S}},\bm{\mathcal{C}})$-weakly superior for \eqref{eq:PG}, then  $\x$ is weakly efficient for \eqref{eq:P}. 
		\item $\x$ is $(\bm{\mathcal{S}},\bm{\mathcal{S}},\bm{\mathcal{C}})$-strictly superior for \eqref{eq:PG} if and only if $\x$ is strictly efficient for \eqref{eq:P}. 
	\end{enumerate}
\end{proposition}

\begin{proof}
	\begin{enumerate}[label=(\alph*)]
		\item A vector $\x\in X$ is  $(\bm{\mathcal{S}},\bm{\mathcal{S}},\bm{\mathcal{C}})$-superior for \eqref{eq:PG} if and only if
		there is no $\bar{\x}\in X$ such that $\f_i(\bar{\x})\leqq \f_i(\x)$ for all $i\in\bm{\mathcal{S}}$ and there exists \emph{at least one} subsystem $j\in\bm{\mathcal{S}}$ such that $\f_j(\bar{\x})\leqslant \f_j(\x)$. Equivalently, there is no $\bar{\x}\in X$ such that $\f(\bar{\x})\leqslant \f(\x)$, i.e., $\x$ is efficient for \eqref{eq:P}.
		\item A vector $\x\in X$ is  $(\bm{\mathcal{S}},\bm{\mathcal{S}},\bm{\mathcal{C}})$-weakly superior for \eqref{eq:PG} if and only if
		there is no $\bar{\x}\in X$ for which \emph{in all} subsystems $i\in\bm{\mathcal{S}}$ it holds that $\f_i(\bar{\x})\leqslant \f_i(\x)$. This implies that there is no $\hat{\x}\in X$ for which \emph{in all} subsystems $i\in\bm{\mathcal{S}}$ it holds that $\f_i(\hat{\x})< \f_i(\x)$, i.e., for which $\f(\hat{\x})<\f(\x)$, and thus $\x$ is weakly efficient for \eqref{eq:P}.

		\item The claim follows directly from the definition. \qedhere
	\end{enumerate}
\end{proof}

The converse of Proposition \ref{lemma:superiority_versus_efficieny}(b) does not hold in general.
\begin{example}\label{ex:lemma:superiority_versus_efficieny} Consider a complex system \eqref{eq:P} with only one subsystem for which the weakly efficient set differs from the efficient set, e.g., $X=X_1=\{\x\in\mathbb{R}^2 : 0\leq x_i\leq 1,\, i=1,2\}$ and $\f(\x)=\f_1(\x)=(x_1,x_2)$, and thus $\f(X)=[0,1]\times[0,1]$, $\wSup(1,1,\emptyset)=E(P)=\{\0\}$ and 
	$wE(P)=(\{0\}\times[0,1])\cup([0,1]\times\{0\})$. See Figure \ref{fig:lemma:superiority_versus_efficieny} for an illustration.
\end{example}
\begin{figure}
	\centering
\begin{tikzpicture}[scale=0.75]
\draw [->] (3,0) -- (3.5,0);
\draw [->] (0,3) -- (0,3.5);
\draw (0,0) -- (0,-0.08);
\draw (1.5,0) -- (1.5,-0.08);
\node [below] at (1.5,-0.1) {$0.5$};
\draw (3,0) -- (3,-0.08);
\node [below] at (3,-0.1) {$1$};
\node [below] at (3.5,-0.1) {$x_1$};
\draw (0,0) -- (-0.08,0);
\draw (0,1.5) -- (-0.08,1.5);
\node [left] at (0,1.5) {$0.5$};
\draw (0,3) -- (-0.08,3);
\node [left] at (0,3) {$1$};
\node [left] at (0,3.5) {$x_2$};
\path [fill=lightgray] (0,0) -- (3,0) -- (3,3) -- (0,3) -- (0,0);
\draw  (3,0) -- (3,3) -- (0,3);
\draw [fill=white] (0,0) circle (.5ex);
\draw [thick,dash pattern={on 7pt off 2pt on 1pt off 3pt}]  (3,0) -- (0,0) -- (0,3);
\node at (2.5,2.5) {$X$};
\end{tikzpicture}
	\caption{\label{fig:lemma:superiority_versus_efficieny} Visualization of Example~\ref{ex:lemma:superiority_versus_efficieny}. $wE(P)$ is marked with a dashed pattern and $\wSup(1,1,\emptyset)$ is marked with a white circle.}
\end{figure}
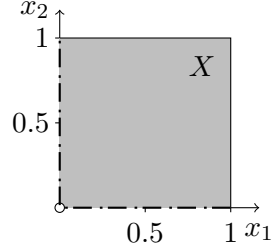

In the presence of all subsystems and a proper subset of their objective functions, superiority for \eqref{eq:PG} implies efficiency for \eqref{eq:P}. 

\begin{proposition}\label{lemma:subsystem_efficieny}
	Let $\x\in X$, i.e., \(\x\in\mathbb{R}^{|\bm{\mathcal{V}|}}\) is system valid, and consider the objective functions of a subset of the subsystems $\emptyset \neq F\subset\bm{\mathcal{S}}$.  Then the following holds: If $\x$ is 
	\begin{enumerate}[label=(\alph*)]
		\item $(F,\bm{\mathcal{S}},\bm{\mathcal{C}})$-weakly superior for \eqref{eq:PG}, then  $\x$ is weakly efficient for \eqref{eq:P}. 
		\item $(F,\bm{\mathcal{S}},\bm{\mathcal{C}})$-strictly superior for \eqref{eq:PG}, then $\x$ is strictly efficient for~\eqref{eq:P}. 
	\end{enumerate}
\end{proposition}

\begin{proof}
	\begin{enumerate}[label=(\alph*)]
		
		\item Let $\x$ be $(F,\bm{\mathcal{S}},\bm{\mathcal{C}})$-weakly superior for \eqref{eq:PG}. Then, there does not exist  \(\bar{\x}\in X_{\bm{\mathcal{S}},\bm{\mathcal{C}}}\) such that  \( \f(\bar{\x})\preceq_{(F,\bm{\mathcal{S}},\bm{\mathcal{C}})}  \f(\x)\). Suppose \(\x\) is not weakly efficient for \eqref{eq:P}. Then, there exists  \(\bar{\x} \in X_{\bm{\mathcal{S}},\bm{\mathcal{C}}}\) with \( \f(\bar{\x}) < \f(\x)\). Equivalently, there exists \(\bar{\x} \in X_{\bm{\mathcal{S}},\bm{\mathcal{C}}}\) such that \( \f_i(\bar{\x}) < \f_i(\x) \) for all  \(i \in \{1,\ldots,|\bm{\mathcal{S}}|\}\). This implies that there exists \(\bar{\x} \in X_{\bm{\mathcal{S}},\bm{\mathcal{C}}}\) with \(\f_i(\bar{\x}) < \f_i(\x) \) for all  \(i \in F\) in contradiction to $\x$ being $(F,\bm{\mathcal{S}},\bm{\mathcal{C}})$-superior for \eqref{eq:PG}.
		
		\item The claim follows directly from the definition. \qedhere
	\end{enumerate}
\end{proof}

Note that the converse of the statements in Proposition \ref{lemma:subsystem_efficieny} do not hold in general. We obtain additional results for all subsystems and one objective function.

\begin{proposition}\label{prop:leadstosuperior}
	Let $\x\in X$ be $(i,\bm{\mathcal{S}},\bm{\mathcal{C}})$-superior for \eqref{eq:PG} for all \(i=1,\ldots,|\bm{\mathcal{S}}|\) (i.e., \(\x\in X\) is efficient w.r.t.\ the subsystem \eqref{eq:subsystemi} for all \(i=1,\ldots,|\bm{\mathcal{S}}|\) ). 
	Then, \(\x\) is efficient for \eqref{eq:P}.
\end{proposition}
\begin{proof}
	Since \(\x\in X\), \(\x\) is feasible for \eqref{eq:P}. We assume that there exists some \(\bar{\x}\in X: \f(\bar{\x})\leqslant \f(\x)\). Then \(\f_i(\bar{\x}
	) \leqq \f_i(\x
	)\) for all \(i=1,\ldots,|\bm{\mathcal{S}}|\) and \(\f_i(\bar{\x}
	)\leqslant \f_i(\x
	)\) for some \(i\in\{1,\ldots,|\bm{\mathcal{S}}|\}\). 
	Thus, $\x$ is not $(i,\bm{\mathcal{S}},\bm{\mathcal{C}})$-superior, a contradiction. 
	Note that \(\bar{\x}\in X\) implies \(\bar{\x}_{s_i}\in X_i\), so \(\x_{s_i}\) is also not  efficient w.r.t.\ the subsystem \eqref{eq:subsystemi} in this case.\qedhere 
\end{proof}

Note that since Proposition \ref{prop:leadstosuperior} makes the stronger assumption that the solution $\x\in X$ is subsystem superior for \emph{all} subsystems, we obtain superiority for the AiO problem. See also Proposition \ref{prop:FsupToSsup} and Example \ref{ex:prop:FsupToSsup} for comparison.

\begin{proposition}\label{lemma:strictsuperiority}
	Let $\x\in X$ be $(i,\bm{\mathcal{S}},\bm{\mathcal{C}})$-strictly superior for \eqref{eq:PG} for at least one \(i\in\{1,\ldots,|\bm{\mathcal{S}}|\}\).   
	Then, \(\x\) is efficient for \eqref{eq:P}.
\end{proposition}

\begin{proof}
	Assume that \(\x\in X\) is not efficient for \eqref{eq:P}. Then there exists a solution $\bar{\x}\in X$ such that $\f(\bar{\x})\leqslant\f(\x)$. Hence, $(\f_1(\bar{\x}_{s_1}),\ldots,\f_{|\bm{\mathcal{S}}|}(\bar{\x}_{s_{|\bm{\mathcal{S}}|}}))\leqslant (\f_1(\x_{s_1}),\ldots,\f_{|\bm{\mathcal{S}}|}(\x_{s_{|\bm{\mathcal{S}}|}}))$ and $\bar{\x}\neq\x$, which implies that $\f_i(\bar{\x}_{s_i})\leqq \f_i(\x_{s_i})$ for all \(i\in\{1,\ldots,|\bm{\mathcal{S}}|\}\) and $\bar{\x}\neq \x$. 
	We can conclude that $\x$ is not $(i,\bm{\mathcal{S}},\bm{\mathcal{C}})$-strictly superior.\qedhere
\end{proof}


%
%
%
%
%
%
%
%

\subsection{Existence of Superior Solutions}

\noindent Due to the close relationship between system superiority and classical efficiency for the AiO problem~\eqref{eq:P} (see Proposition \ref{lemma:superiority_versus_efficieny} above), general existence results can be transferred from the corresponding results in the field of multiobjective optimization. For example, if the system valid set $X$ is compact and all objective functions are continuous, then the system superior set is nonempty.  Similarly, if $Y=\f(X)$ has a compact section $Y^0:=\{\y\in Y\, :\, \y\leqq \y^0\}$ for some $\y^0\in Y$, then again the system superior set is nonempty. See, for example, \cite{Ehrgott} for a survey of these and other existence results in multiobjective optimization.

In the following, we will discuss conditions under which the existence of superior solutions on the subsystem level implies the existence of superior solutions on the system level.

\begin{proposition}\label{prop:existsonesubsystem}
	Suppose that there is at least one subsystem $s_i\in\bm{\mathcal{S}}$ with nonempty $(i,\bm{\mathcal{S}},\bm{\mathcal{C}})$-weakly superior set, i.e.,
	$\exists i\in \{1,\ldots,|\bm{\mathcal{S}}|\}\,:\, \wSup(i,\bm{\mathcal{S}},\bm{\mathcal{C}}) \neq \emptyset.$ Then
	$\wSup(\bm{\mathcal{S}},\bm{\mathcal{S}},\bm{\mathcal{C}})\neq\emptyset$.
\end{proposition}

\begin{proof}
	Follows from Proposition \ref{lemma:subsystem_efficieny}.
\end{proof}

Conversely, even if for all subsystems $s_i\in\bm{\mathcal{S}}$ the 
$(i,\bm{\mathcal{S}},\bm{\mathcal{C}})$-superior set is nonempty, i.e.,
$\forall i\in \{1,\ldots,|\bm{\mathcal{S}}|\}\,:\, \Sup(i,\bm{\mathcal{S}},\bm{\mathcal{C}}) \neq \emptyset,$
this does in general \emph{not} imply that
$\Sup(\bm{\mathcal{S}},\bm{\mathcal{S}},\bm{\mathcal{C}})\neq\emptyset$.
We provide a simple example below.

\begin{figure}
	\centering
\begin{tikzpicture}[scale=1.5]\small
\path [fill=lightgray] (0,0.5) -- (0,2) -- (2,2) -- (2,0) -- (0.5,0);
\draw [<->] (0,2.2) -- (0,0) -- (2.5,0);
\draw (0,0) -- (-0.08,0);
\draw (0,1) -- (-0.08,1);
\draw (0,2) -- (-0.08,2);
\node [left] at (0,2.2) {$f_2(x_1,x_2)$};
\node [left] at (-0.05,2) {2};
\node [left] at (-0.05,1) {1};
\draw (0,0) -- (0,-0.08);
\draw (1,0) -- (1,-0.08);
\draw (2,0) -- (2,-0.08);
\node [below right] at (2.4,0) {$f_1(x_1,x_2)$};
\node [] at (1.25,1.5) {$\f(X_{\{1,2\},\emptyset})$};
\draw[dashed] (0,0.5) -- (0.5,0);
\draw (0,2) -- (2,2);
\draw (2,0) -- (2,2);
\node[circle, draw=black, inner sep=2pt] (0) at (0.5,0){};
\node[circle, draw=black, inner sep=2pt] (0) at (0,0.5){};
\draw[ultra thick] (0.5,0) -- (2,0);
\draw[ultra thick] (0,0.5) -- (0,2);
\node[left] at (0,1.33){$\wSup(1,\{1,2\},\emptyset)$};
\node[below] at (1.33,0){$\wSup(2,\{1,2\},\emptyset)$};
\end{tikzpicture}
	
	\caption{Illustration of the image $\f(X)$ of the AiO problem of Example \ref{example:nonexistence} with the subsystem superior sets in the objective space. \label{fig:examplefornonexistence} }
\end{figure}
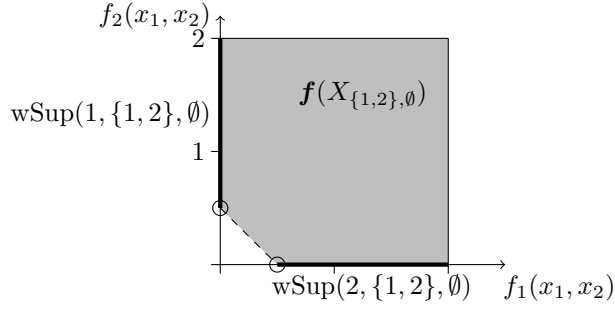

\begin{example}\label{example:nonexistence}
	The following AiO problem has two variables $x_1$ and $x_2$ (which are  both global variables in the considered decomposition) and consists of two subsystems $s_1$ and $s_2$ that are both single-objective and share the same feasible set:
	\begin{equation*}
	(P_1)\,\begin{aligned}
	\min \;& f_1(x_1,x_2) = x_1 \\
	\text{s.t.} \; &  0\leq x_1,x_2 \leq 2\\
	& x_1+x_2 > \tfrac{1}{2}
	\end{aligned}
	\hspace{2cm}
	(P_2)\,\begin{aligned}
	\min \;& f_2(x_1,x_2) = x_2 \\
	\text{s.t.} \; &  0\leq x_1,x_2 \leq 2\\
	& x_1+x_2 > \tfrac{1}{2}
	\end{aligned}
	\end{equation*}
	In this case, $\Sup(1,\{1,2\},\emptyset)=\{(0,t)\in\R^2\,:\, \frac{1}{2}<t\leq 2\}$ and
	$\Sup(2,\{1,2\},\emptyset)=\{(t,0)\in\R^2\,:\, \frac{1}{2}<t\leq 2\}$, but
	$\Sup(\{1,2\},\{1,2\},\emptyset)=\emptyset$. See Figure \ref{fig:examplefornonexistence} for an illustration.
\end{example}

\subsection{Illustrative Example}

\noindent Example~\ref{example:illustrativeexample} is used to illustrate some properties of complex systems.
\begin{example}\label{example:illustrativeexample} 
	We consider an AiO problem with four variables $x_1,x'_1,x_2,x_3\in\R$ decomposed into two subsystems $s_1$ and $s_2$ as follows:
	\begin{equation}\notag
	(P_1)\,\begin{aligned}
	\vmin \;& \f_1(x_1,x_2) = (-x_1-x_2,\, 2x_1+x_2) \\
	\text{s.t.} \; & x_1+x_2 \geq \tfrac{1}{2}\\
	& 0\leq x_1 \leq 2\\
	& 0\leq x_2 \leq 2
	\end{aligned}
	\hspace{1cm}
	(P_2)\,\begin{aligned}
	\vmin \;& \f_2(x'_1,x_3) = (-x'_1,\, -x'_1-x_3) \\
	\text{s.t.} \; &  x'_1+x_3 \leq \tfrac{7}{2}\\
	& 0\leq x'_1\leq 3\\
	& 0\leq x_3\leq 2
	\end{aligned}
	\end{equation}
	The two subsystems are linked by
	one linking constraint given by $$c_1(x_1,\, x'_1)=x_1-x'_1=0.$$
	Thus, $\bm{\mathcal{S}}=\{\sigma_1,\sigma_2\}$ and $\bm{\mathcal{C}}=\{\kappa_1\}$. The complex system graph of this example problem is illustrated in Figure \ref{fig:illustrative_example}.
\end{example}

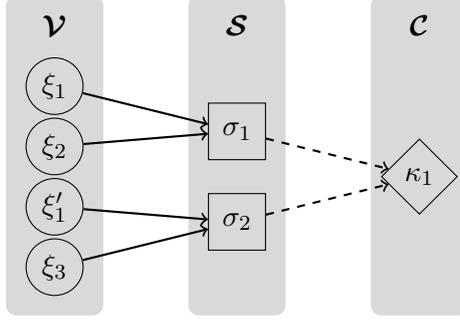
\begin{figure}
	\centering
\begin{tikzpicture}[scale=0.8]
\node[] (Variables) at (0,11){$\bm{\mathcal{V}}$};
\node[] (Subsystems) at (3,11){$\bm{\mathcal{S}}$};
\node[] (Linking) at (6,11){$\bm{\mathcal{C}}$};
%
\node [draw, circle,minimum size=0.75cm,inner sep=2pt] (xi3) at (0,7) {$ \xi_3 $};
\node [draw, circle,minimum size=0.75cm,inner sep=2pt] (xi1prime) at (0,8) {$ \xi'_1 $};
\node [draw, circle,minimum size=0.75cm,inner sep=2pt] (xi2) at (0,9) {$ \xi_2 $};
\node [draw, circle,minimum size=0.75cm,inner sep=2pt] (xi1) at (0,10) {$ \xi_1 $};
\node[draw, rectangle,minimum size=0.75cm] (S2) at (3,7.75) {$ \sigma_2 $};
\node[draw, rectangle,minimum size=0.75cm] (S1) at (3,9.25) {$ \sigma_1 $};
\node[] (Slast) at (3,6.7) {};
\node[draw, diamond,minimum size=1cm,inner sep=1pt] (Link1) at (6,8.5) {$ \kappa_1 $};
\node[] (Link3) at (6,6.7) {};

\draw[->,thick] (xi1)--(S1);
\draw [->,thick] (xi1prime)--(S2);
\draw [->,thick] (xi2)--(S1);
\draw [->,thick] (xi3)--(S2);
\draw [->,dashed,thick] (S1)--(Link1);
\draw [->,dashed,thick] (S2)--(Link1);
\begin{scope}[on background layer]
\fill [fill=black!15,rounded corners] (-0.8,6.2) rectangle (0.8,11.5);
\fill [fill=black!15,rounded corners] (2.2,6.2) rectangle (3.8,11.5);
\fill [fill=black!15,rounded corners] (5.2,6.2) rectangle (6.8,11.5);
\end{scope}
\end{tikzpicture}
	\caption{\label{fig:illustrative_example} Complex system graph \(G\) for Example \ref{example:illustrativeexample}.}
\end{figure}

Example~\ref{example:illustrativeexample} describes a four dimensional problem that can be illustrated in three dimensions since $x_1=x'_1$ for all system valid solutions. Note that the sets of $(1,\bm{\mathcal{C}})$-valid and $(1,\emptyset)$-valid solutions are not the same in $\mathbb{R}^4$, but their projections onto the $(x_1,\, x_2)$-space are the same.

Figure~\ref{fig:illustrativeexample} illustrates $(S,C)$-valid sets for different combinations of $S\subseteq\bm{\mathcal{S}}$ and $C\subseteq\bm{\mathcal{C}}$ together with the corresponding superior sets. Note that considering an individual subproblem provides not much insight for the AiO problem at hand. In particular, Figure~\ref{fig:illustrativeexample} (a) and (b) do not reveal that $\Sup(1,\bm{\mathcal{S}},\bm{\mathcal{C}})$ and $\Sup(2,\bm{\mathcal{S}},\bm{\mathcal{C}})$ share one common point (as can be seen in part (c) of the figure). Even the knowledge of the entire sets $\Sup(1,\bm{\mathcal{S}},\bm{\mathcal{C}})$ and $\Sup(2,\bm{\mathcal{S}},\bm{\mathcal{C}})$ might in general not be enough to reconstruct the set $\Sup(\bm{\mathcal{S}},\bm{\mathcal{S}},\bm{\mathcal{C}})$ (as can be seen in part (d) of the figure).

\begin{figure}[htb]
  \centering
	\begin{subfigure}[]
		\centering
		\includegraphics[scale=0.12]{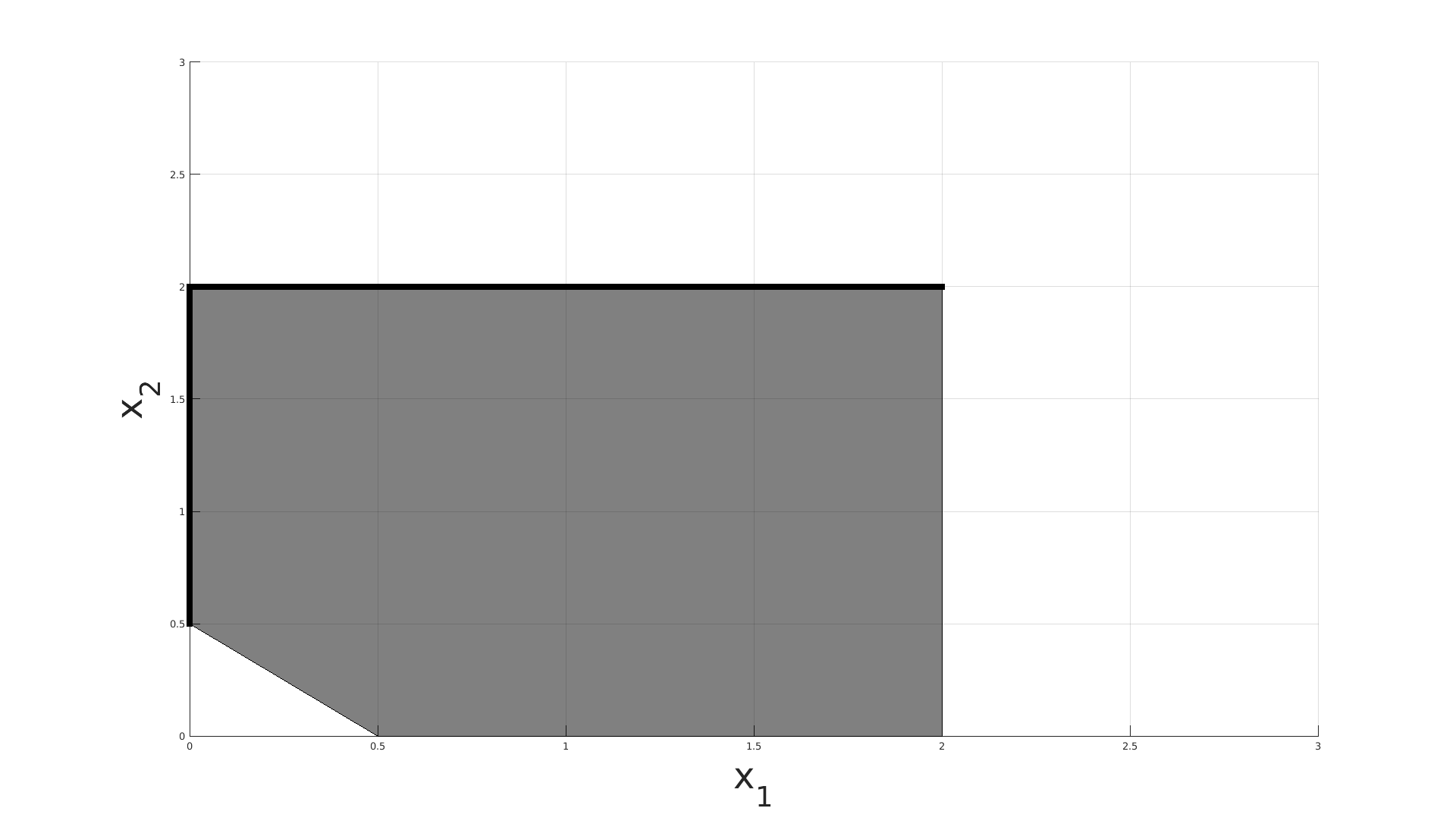} 
	\end{subfigure}
	\begin{subfigure}[]
		\centering
		\includegraphics[scale=0.12]{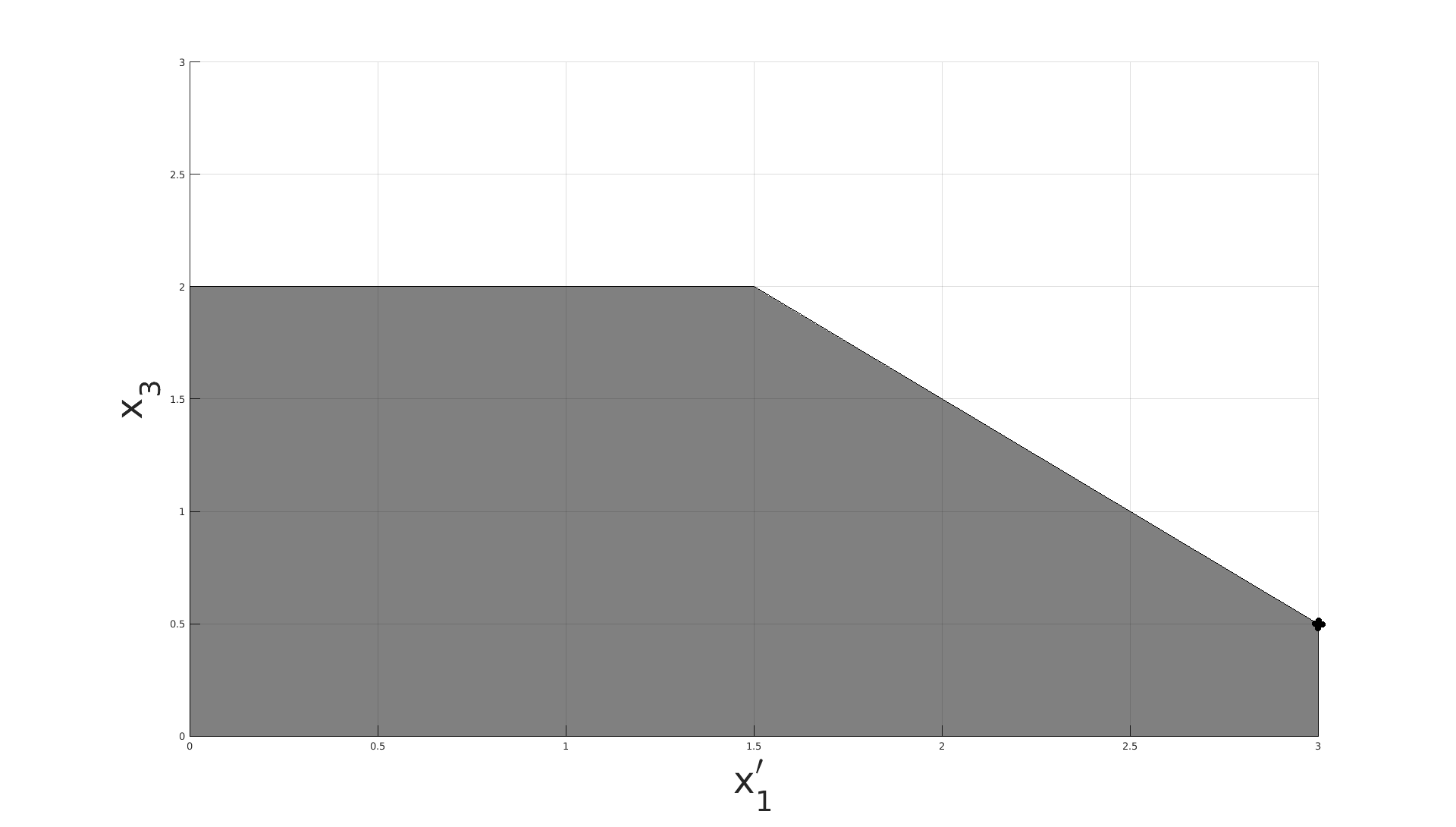}
	\end{subfigure}
	\begin{subfigure}[]
		\centering
		\includegraphics[scale=0.12]{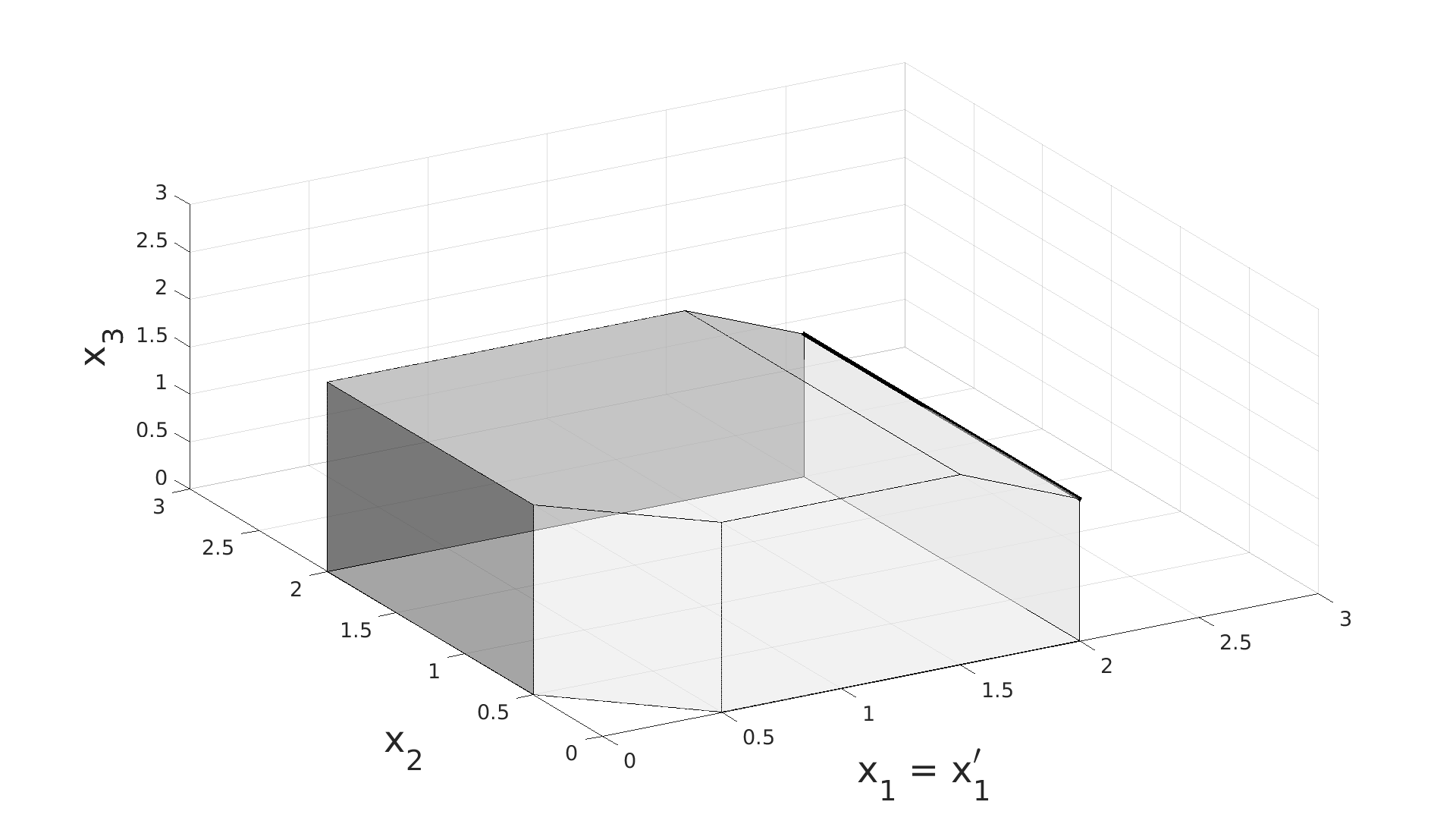}
	\end{subfigure}
	\begin{subfigure}[]
		\centering
		\includegraphics[scale=0.12]{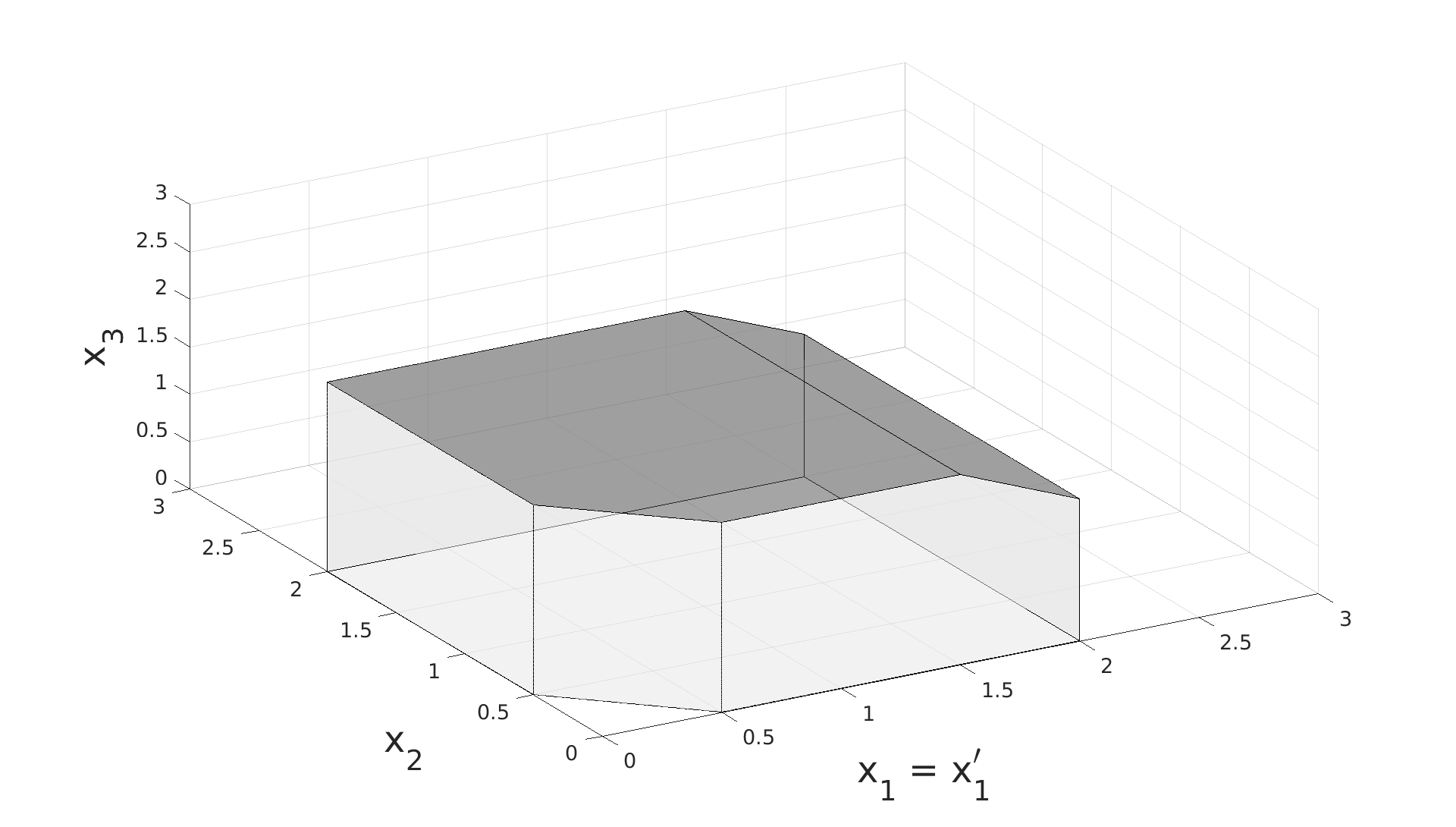}
	\end{subfigure}
	\caption{(a) Illustration of the $(1,\emptyset )$-valid set of Example~\ref{example:illustrativeexample}. The set $\Sup(1,1,\emptyset )$ is highlighted by bold black lines.
		(b) Illustration of the $(2,\emptyset )$-valid set of Example~\ref{example:illustrativeexample}. The set $\Sup(2,2,\emptyset)$ projects onto a single point that is highlighted by a black dot.
		(c) Illustration of the $(\bm{\mathcal{S}},\, \bm{\mathcal{C}})$-valid set of Example~\ref{example:illustrativeexample}. The set $\Sup(1,\bm{\mathcal{S}},\bm{\mathcal{C}})$ is highlighted in dark gray and the set $\Sup(2,\bm{\mathcal{S}},\bm{\mathcal{C}})$ is highlighted by a bold black line.
		(d) Illustration of the $(\bm{\mathcal{S}},\, \bm{\mathcal{C}})$-valid set of Example~\ref{example:illustrativeexample} with the set $\Sup(\bm{\mathcal{S}},\bm{\mathcal{S}},\bm{\mathcal{C}})$ highlighted in dark gray.\label{fig:illustrativeexample}}
\end{figure}

\section{Achieving Superiority}\label{sec:achsuperiority}

\subsection{Obtaining Lower Bounds}

\noindent The decomposition of an AiO problem \eqref{eq:P} based on a complex system graph $G$ gives rise to a generalization of the concept of ideal outcome vectors. We make use of the objective space images 
of the efficient sets to the individual  subsystems \eqref{eq:subsystemi} and call them \emph{subsystem level} ideal sets because they represent the best objective values each subsystem can achieve with no interaction with other subsystems. Similarly, the ideal point to every subsystem \eqref{eq:subsystemi} provides a lower bound on the performance of subsystem $s_i$.  The same concepts can be defined for every subsystem at the \emph{system level} when its interaction with other subsystems is considered. Obviously, the subsystem level ideal set 
is a lower bound for the  system level ideal set, 
which is illustrated in Figure \ref{fig:ideal}. The mutual location of these sets may provide a measure of the contribution of subsystem $s_i$ to the interaction with the other subsystems in the units of its  decayed performance.
A definition of these concepts is given below.


\begin{definition}[Lower bounds on superior objective values]
	Consider an AiO problem \eqref{eq:P} and its decomposition \eqref{eq:PG}.
	\begin{enumerate}[label=(\alph*)]
		\item An objective vector $\y\in\mathbb{R}^{p}$ is called \emph{subsystem level ideal} if all of its subvectors $\y_{s_i}$, $i\in\{1,\dots,|\bm{\mathcal{S}}|\}$ are images of $(i,i,\emptyset)$-superior solutions, i.e., of feasible and efficient solutions for subsystem $s_i$. The \emph{subsystem level ideal set} is  given by 
		$Y^{ssI} := \f_1(\Sup(1,1,\emptyset)) \times \cdots \times \f_{|\Sys|}(\Sup(|\bm{\mathcal{S}}|,|\bm{\mathcal{S}}|,\emptyset))$.
		\item Let $\y_i^{ssI}$ be the individual subsystem ideal point for all $i\in\{1,\dots,|\bm{\mathcal{S}}|\}$, i.e., $y_{i,k}^{ssI}= \min \{f_{ik} (\x_{s_i}):  \x_{s_i} \in X_i \}$, $k = 1 ,\ldots, p_i$, is the ideal point w.r.t.\ images of the $(i,i,\emptyset)$-superior solutions. Then $\y^{ssI}:=(\y_1^{ssI},\dots,\y_{|\Sys|}^{ssI})$ is called the \emph{subsystem level ideal point}.
		\item An objective vector
		$\y\in\mathbb{R}^{p}$ is called \emph{system level ideal} if all of its subvectors $\y_{s_i}$, $i\in\{1,\dots,|\bm{\mathcal{S}}|\}$ are images of $(i,\Sys,\bm{\mathcal{C}})$-superior solutions, i.e., system valid solutions that are efficient w.r.t.\ subsystem $s_i$.
		The \emph{system level ideal set} is given by $Y^{sI}:=\f_1(\Sup(1,\bm{\mathcal{S}},\bm{\mathcal{C}})) \times \cdots \times \f_{|\Sys|}(\Sup(|\bm{\mathcal{S}}|,\bm{\mathcal{S}},\bm{\mathcal{C}}))$.
		\item Let $\y_i^{sI}$ be the system ideal points for all $i\in\{1,\dots,|\bm{\mathcal{S}}|\}$, i.e., $\y_{i,k}^{sI}= \min \{f_{ik} (\x_{s_i}):$ $ \x_{s_i} \in X_{\bm{\mathcal{S}},\bm{\mathcal{C}}} \}$, $k = 1 ,\ldots, p_i$, is the ideal point w.r.t.\ images of the $(i,\bm{\mathcal{S}},\bm{\mathcal{C}})$-superior solutions. Then $\y^{sI}:=(\y_1^{sI},\dots,\y_{|\Sys|}^{sI})$ is called the \emph{system level ideal point}.
	\end{enumerate}
\end{definition}

Note that $\y^{ssI}$ is a worse lower bound than $\y^{sI}$ but, depending on the specific setup, may be easier to compute since only a single objective is taken into account.


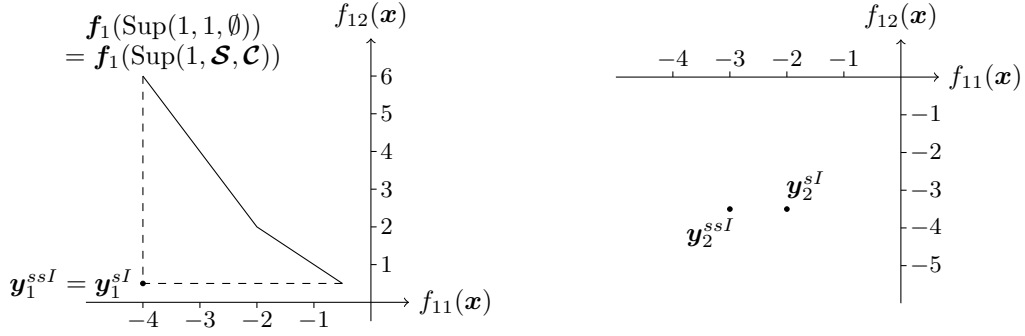
\begin{figure}
	\centering\small
\begin{tikzpicture}[scale=0.5]
\draw[->] (-0.5,0) -- (8,0);
\node [right] at (8,0) {\small$f_{11}(\x)$};
\draw (5.5,0) -- (5.5,-0.15);
\node [below] at (5.5,0) {\footnotesize$-1$};
\draw (4,0) -- (4,-0.15);
\node [below] at (4,0) {\footnotesize$-2$};
\draw (2.5,0) -- (2.5,-0.15);
\node [below] at (2.5,0) {\footnotesize$-3$};
\draw (1,0) -- (1,-0.15);
\node [below] at (1,0) {\footnotesize$-4$};
\draw[->] (7,-0.5) -- (7,7); 
\node [above] at (7,7) {\small$f_{12}(\x)$};
\draw (7,1) -- (7.15,1);
\node [right] at (7,1) {\footnotesize$1$};
\draw (7,2) -- (7.15,2);
\node [right] at (7,2) {\footnotesize$2$};
\draw (7,3) -- (7.15,3);
\node [right] at (7,3) {\footnotesize$3$};
\draw (7,4) -- (7.15,4);
\node [right] at (7,4) {\footnotesize$4$};
\draw (7,5) -- (7.15,5);
\node [right] at (7,5) {\footnotesize$5$};
\draw (7,6) -- (7.15,6);
\node [right] at (7,6) {\footnotesize$6$};
\draw (6.25,0.5) -- (4,2)  -- (1,6);
\node [above] at (1.8,6.7) {\small$\f_1(\Sup(1,1,\emptyset))$};
\node [above] at (1.8,5.9) {\small$=\f_1(\Sup(1,\Sys,\bm{\mathcal{C}}))$};
\draw [dashed] (6.25,0.5) -- (1,0.5) -- (1,6);
\draw [fill] (1,0.5) circle [radius=0.06];
\node [left] at (1,0.5) {\small$\y^{ssI}_1=\y^{sI}_1$};
\end{tikzpicture}\hspace*{1.5cm}\begin{tikzpicture}[scale=0.5]
\draw[->] (-0.5,6) -- (8,6);
\node [right] at (8,6) {\small$f_{11}(\x)$};
\draw (5.5,6) -- (5.5,6.15);
\node [above] at (5.5,6) {\footnotesize$-1$};
\draw (4,6) -- (4,6.15);
\node [above] at (4,6) {\footnotesize$-2$};
\draw (2.5,6) -- (2.5,6.15);
\node [above] at (2.5,6) {\footnotesize$-3$};
\draw (1,6) -- (1,6.15);
\node [above] at (1,6) {\footnotesize$-4$};
\draw[->] (7,0) -- (7,7); 
\draw[white] (7,-0.9) -- (7,0); 
\node [above] at (7,7) {\small$f_{12}(\x)$};
\draw (7,1) -- (7.15,1);
\node [right] at (7,1) {\footnotesize$-5$};
\draw (7,2) -- (7.15,2);
\node [right] at (7,2) {\footnotesize$-4$};
\draw (7,3) -- (7.15,3);
\node [right] at (7,3) {\footnotesize$-3$};
\draw (7,4) -- (7.15,4);
\node [right] at (7,4) {\footnotesize$-2$};
\draw (7,5) -- (7.15,5);
\node [right] at (7,5) {\footnotesize$-1$};
\node [below] at (2,2.5) {\small$\y^{ssI}_2$};
\draw [fill] (2.5,2.5) circle [radius=0.06];
\node [above] at (4.5,2.5) {\small$\y^{sI}_2$};
\draw [fill] (4,2.5) circle [radius=0.06];
\end{tikzpicture}\hfill
	\caption{The system level ideal points $\y^{sI}_i$ and the subsystem level ideal points $\y^{ssI}_i$ for both subsystems of example~\ref{example:illustrativeexample}. Note that $\f_2(\Sup(2,2,\emptyset))=\y^{ssI}_2$ and $\f_2(\Sup(2,\Sys,\bm{\mathcal{C}}))=\y^{sI}_2$ hold.\label{fig:ideal}}
\end{figure}

\begin{proposition}
	Consider an AiO problem \eqref{eq:P} and its reformulation \eqref{eq:PG}, and let $\y^I$ be the ideal point of the AiO problem \eqref{eq:P}. Then the following holds:
	\begin{enumerate}[label=(\alph*)]
		\item $Y^{sI}\subseteq Y^{ssI}+\R^p_{\geqq}$.
		\item $\y^{ssI}\leqq \y^I = \y^{sI}$ and $\f(X)\subseteq \y^{sI}+\R^p_{\geqq}$.
	\end{enumerate}
\end{proposition}

\begin{proof}
	\begin{enumerate}[label=(\alph*)]
		\item Let $\y=(\y_1,\dots,\y_{|\Sys|})\in Y^{sI}$. By definition, $\y_i\in \f_i(\Sup(i,\Sys,\bm{\mathcal{C}}))$  for all $i\in\{1,\dots,|\bm{\mathcal{S}}|\}$. Since $X=X_{\Sys, \bm{\mathcal{C}}}\subseteq X_{i,\emptyset}$, we have that $\f_i(\Sup(i,\Sys,\bm{\mathcal{C}}))\subseteq \f_i(\Sup(i,i,\emptyset))+\R^{|\pred{\sigma_i}|}_{\geqq}$ for all $i\in\{1,\dots,|\bm{\mathcal{S}}|\}$, and the result follows.\label{proof:ideal1}
		\item From \ref{proof:ideal1}, we immediately have that $\y^{ssI}\leqq\y^{sI}$. Moreover, since $\y_i^{sI}\leqq \f_i(\x)$ for all $\x\in X=X_{\Sys, \bm{\mathcal{C}}}$, we can conclude that $\y^{sI} \leqq \f(\x)$ for all $\x\in X$. To see that $\y^I=\y^{sI}$, note that the minimization of the individual objective functions is performed over the same feasible set $X=X_{\Sys, \bm{\mathcal{C}}}$.\qedhere
	\end{enumerate}
\end{proof}

Note that unless the AiO problem \eqref{eq:P} has a very specific structure (e.g., it is fully decomposable, see Section \ref{subsec:independentsubsystems} below), there are no system valid solutions that map onto the subsystem ideal set or the system ideal set. 
Moreover, $\y^{ssI}\neq\y^{sI}$ in general.

\subsection{Block Diagonal Systems with Independent Subsystems}\label{subsec:independentsubsystems}



\noindent The purpose of this section is twofold. We first present a specific complex system that is  decomposable into a collection of subsystems such that their superior sets constitute the superior set of the complex system. Because of this simple but very useful relationship, this type of complex system appears ``easy'' and suggests a desirable structure for complex systems in general. We next show that the decomposition of this complex system is not unique and does not always lead to the same subsystems. We illustrate the effect of different decompositions of the same complex system.

We consider the AiO problem of the form
\begin{equation*}\tag{$P_I$}\label{eq:Pwithindependentsubsystems}
\begin{aligned}
\vmin \;& (\f_{1}(\x_{s_1}) , \f_{2}(\x_{s_2}) , \dots , \f_{|\bm{\mathcal{S}}|}(\x_{s_{|\bm{\mathcal{S}}|}}))\\
\text{s.t.} \; &  \x_{s_i}\in X_i,\quad i=1,\dots,|\bm{\mathcal{S}}|,
\end{aligned}
\end{equation*}
where the feasible set $X=X_1\times\cdots\times X_{|\bm{\mathcal{S}}|}$, i.e., there are no linking constraints combining variables from any pair of subsets $X_i$ and $X_j$ with $i\neq j$.  

A complex system graph $G$ depicted in Figure \ref{fig:blockdiagonalsystem}   is associated with problem \eqref{eq:Pwithindependentsubsystems} and results in the complex system $P_I(G)$.
Note that graph $G$ consists of $|\bm{\mathcal{S}}|$ independent subgraphs $G_i$, $i=1,\dots,|\bm{\mathcal{S}}|$, each of which represents a subsystem $s_i$. 
Since there are no linking constraints ($\bm{\mathcal{C}}=\emptyset$), the node set of $G$ is given by $\bm{\mathcal{V}}\cup\bm{\mathcal{S}}$ and the arc set of $G$ is given by $R(\bm{\mathcal{V}},\bm{\mathcal{S}})$. Moreover, since all variables are local variables,
all variable nodes $\xi_k\in\pred(\sigma_i)$ associated with subsystem $s_i$ are only connected to the subsystem node $\sigma_i$ for all $i=1,\dots,|\Sys|$. Hence, each subsystem $s_i$ corresponds to an induced subgraph $G_i$ of $G$ with nodes $\{\sigma_i\}\cup \pred(\sigma_i)$ and edges $\{(\xi_k,\sigma_i) : \xi_k\in\pred(\sigma_i) \}$, and there are no edges between two different subgraphs $G_i$ and $G_j$, $i\neq j$. 

The complex system $P_I(G)$ has the form 
\begin{equation*}\tag{$P_I(G)$}\label{eq:complexsystemPwithindependentsubsystems}
\begin{aligned}
\smin \;& (\f_{1}(\x_{s_1}) , \f_{2}(\x_{s_2}) , \dots , \f_{|\bm{\mathcal{S}}|}(\x_{s_{|\bm{\mathcal{S}}|}}))\\
\text{s.t.} \; &  \x_{s_i}\in X_i,\quad i=1,\dots,|\bm{\mathcal{S}}|.
\end{aligned}
\end{equation*}
The graph implies a decomposition of  $P_I(G)$  into $|\bm{\mathcal{S}}|$ independent subsystems $s_i$, $i=1,\dots,|\bm{\mathcal{S}}|$, i.e., into subsystems that have blocks of objective functions that are defined on disjoint subsets of variables. In effect, the complex system has a block diagonal structure, while every block is a subsystem assuming the form of the following MOP:
\begin{equation}\tag{$P_{I_i}$}\label{eq:independentsubsystemi}
\begin{array}{ll}
\vmin & \f_i(\x_{s_i}) \\
\text{s.t.} & \x_{s_i} \in X_i.
\end{array}
\end{equation}

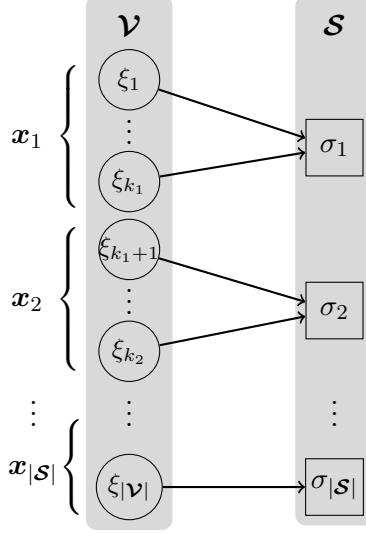
\begin{figure}
	\centering
\begin{tikzpicture}[scale=0.9]
\node[] (Variables) at (0,10.8){$\bm{\mathcal{V}}$};
\node[] (Subsystems) at (3,10.8){$\bm{\mathcal{S}}$};
\node [] (1) at (-1.2,4.3) {\(\x_{|\bm{\mathcal{S}}|\!\!\!}\;\left\{\rule{0pt}{0.7cm}\right.\)};
\node[] (xdots) at (-1.4,5.2) {$ \vdots $};
\node [draw, circle,minimum size=0.8cm,inner sep=2pt] (xilast) at (0,4) {\small$ \xi_{|\bm{\mathcal{V}}|} $};
\node[] (xidots) at (0,5.2) {$ \vdots $};
\node [] (1) at (-1.2,6.8) {\(\x_2\;\left\{\rule{0pt}{1.1cm}\right.\)};
\node [draw, circle,minimum size=0.8cm,inner sep=2pt] (xi4) at (0,6) {\small$ \xi_{k_2}$};
\node[] (xidots2) at (0,6.85) {$ \vdots $};
\node [draw, circle,minimum size=0.8cm,inner sep=2pt] (xi3) at (0,7.5) {\small$\!\!\!\xi_{k_1+1}\!\!\!$};
\node [] (1) at (-1.2,9.2) {\(\x_1\;\left\{\rule{0pt}{1.1cm}\right.\)};
\node [draw, circle,minimum size=0.8cm,inner sep=2pt] (xi2) at (0,8.5) {\small$ \xi_{k_1}$};
\node[] (xidots1) at (0,9.35) {$ \vdots $};
\node [draw, circle,minimum size=0.8cm,inner sep=2pt] (xi1) at (0,10) {\small$ \xi_1 $};
\node[draw, rectangle,minimum size=0.75cm,inner sep=2pt] (S1) at (3,9) {$ \sigma_1 $};
\node[draw, rectangle,minimum size=0.75cm,inner sep=2pt] (S2) at (3,6.6) {$ \sigma_2 $};
\node[] (Sdots) at (3,5.2) {$ \vdots $};
\node[draw, rectangle,minimum size=0.75cm,inner sep=2pt] (Slast) at (3,4) {$ \sigma_{|\bm{\mathcal{S}}|}  $};
\draw[->,thick] (xi1)--(S1);
\draw [->,thick] (xi2)--(S1);
\draw[->,thick] (xi3)--(S2);
\draw [->,thick] (xi4)--(S2);
\draw [->,thick] (xilast)--(Slast);
\begin{scope}[on background layer]
\node[fill=black!15,fit=(Variables) (xilast),rounded corners] {}; 
\node[fill=black!15,fit=(Subsystems) (Slast) ,rounded corners] {}; 
\end{scope}
\end{tikzpicture}
	\caption{\label{fig:blockdiagonalsystem} Complex system graph 
		with independent subsystems}
\end{figure}

In this case, the individual efficient sets of the subsystems can be combined into the $(\Sys,\Sys,\bm{\mathcal{C}})$-superior set of \eqref{eq:Pwithindependentsubsystems} and vice versa. In other words, system level ideal solutions 
are feasible for the AiO problem and thus also AiO efficient. Since $\Sup(i,i,\emptyset)\subseteq\R^{|\bm{\mathcal{V}}|}$, we will use the notation $\Sup(i,i,\emptyset)\bigl|_{X_i}$ to refer to the projection of the subsystem superior set $\Sup(i,i,\emptyset)$ onto the feasible set $X_i$ of subsystem $s_i$, $i=1,\dots,|\Sys|$.

\begin{proposition}\label{prop:Pwithindependentsubsets}
	For a complex system \eqref{eq:Pwithindependentsubsystems} with a decomposition into independent subsystems \eqref{eq:independentsubsystemi} it holds:
	\begin{enumerate}[label=(\alph*)]
		\item $E(P_I) = \Sup(\Sys,\Sys,\bm{\mathcal{C}}) = \Sup(1,1,\emptyset)\bigl|_{X_1} \times \ldots \times \Sup(|\Sys|,|\Sys|,\emptyset)\bigl|_{X_{|\Sys|}} $,
		\item $wE(P_I) \supseteq \wSup(1,1,\emptyset)\bigl|_{X_1} \times \ldots \times \wSup(|\Sys|,|\Sys|,\emptyset)\bigl|_{X_{|\Sys|}} $.
	\end{enumerate}
\end{proposition}

\begin{proof}
	\begin{enumerate}[label=(\alph*)]
		\item The first equality results from Proposition \ref{lemma:superiority_versus_efficieny}(a). We first show that $\Sup(\Sys,\Sys,\bm{\mathcal{C}}) \supseteq \Sup(1,1,\emptyset)\bigl|_{X_1} \times \ldots \times \Sup(|\Sys|,|\Sys|,\emptyset)\bigl|_{X_{|\Sys|}} $.
		Let $\x\in \Sup(1,1,\emptyset)\bigl|_{X_1} \times$ $\ldots \times \Sup(|\Sys|,|\Sys|,\emptyset)\bigl|_{X_{|\Sys|}} $. Thus, $\x=(\x_1,\dots,\x_{|\bm{\mathcal{S}}|})^T\in X_1\times \cdots\times X_{|\bm{\mathcal{S}}|}=X$ is system valid. Suppose that $\x\not\in \Sup(\Sys,\Sys,\bm{\mathcal{C}})$, i.e., there is a system valid solution $\bar{\x}=(\bar{\x}_1,\dots,\bar{\x}_{|\bm{\mathcal{S}}|})^T\in X_1\times \cdots \times X_{|\bm{\mathcal{S}}|}$ such that
		$\f(\bar{\x}) \leqslant \f(\x)$. This implies that 
		$$\f_1(\bar{\x}_1) \leqslant \f_1(\x_1)\quad \vee \quad \f_2(\bar{\x}_2) \leqslant \f_2(\x_2) \quad \vee \quad \cdots \quad \vee \quad \f_{|\bm{\mathcal{S}}|}(\bar{\x}_{|\bm{\mathcal{S}}|}) \leqslant \f_{|\bm{\mathcal{S}}|}(\x_{|\bm{\mathcal{S}}|})$$
		and thus $\x\not\in \Sup(1,1,\emptyset)\bigl|_{X_1} \times \ldots \times \Sup(|\Sys|,|\Sys|,\emptyset)\bigl|_{X_{|\Sys|}} $, a contradiction.
		
		It remains to show that $\Sup(\Sys,\Sys,\bm{\mathcal{C}}) \subseteq \Sup(1,1,\emptyset)\bigl|_{X_1} \times \ldots \times \Sup(|\Sys|,|\Sys|,\emptyset)\bigl|_{X_{|\Sys|}} $.
		Now suppose that $\x\in \Sup(\Sys,\Sys,\bm{\mathcal{C}})$, i.e., $\x=(\x_1,\dots,\x_{|\bm{\mathcal{S}}|})^T\in X_1\times \cdots\times X_{|\bm{\mathcal{S}}|}$. Suppose that $\x_i\not\in \Sup(i,i,\emptyset)\bigl|_{X_i}$ for some $i\in\{1,\dots,|\bm{\mathcal{S}}|\}$. Then there is an $\bar{\x}_i\in X_i$ such that $\f_i(\bar{\x}_i)\leqslant \f_i(\x_i)$. But then $\bar{\x}:=(\x_1,\dots,\x_{i-1},\bar{\x}_i,\x_{i+1},\dots,\x_{|\bm{\mathcal{S}}|})^T\in X_1\times \cdots \times X_{|\bm{\mathcal{S}}|}$ is system valid and satisfies $\f(\bar{\x})\leqslant \f(\x)$, a contradiction. 
		\item Similarly, let $\x\in \wSup(1,1,\emptyset)\bigl|_{X_1} \times \ldots \times \wSup(|\Sys|,|\Sys|,\emptyset)\bigl|_{X_{|\Sys|}}$. This implies that $\x=(\x_1,\dots,\x_{|\bm{\mathcal{S}}|})^T\in X_1\times \cdots\times X_{|\bm{\mathcal{S}}|}=X$ is system valid. Suppose that $\x\not\in wE(P_I)$, i.e., there is a system valid solution $\bar{\x}=(\bar{\x}_1,\dots,\bar{\x}_{|\bm{\mathcal{S}}|})^T\in X_1\times \cdots \times X_{|\bm{\mathcal{S}}|}$ such that
		$\f(\bar{\x}) < 
		\f(\x)$. This implies that 
		$\f_i(\bar{\x}_i) < \f_i(\x_i)$ for all $i\in\{1,\dots,|\Sys|\}$ and thus $\x\not\in \wSup(1,1,\emptyset)\bigl|_{X_1} \times \ldots \times \wSup(|\Sys|,|\Sys|,\emptyset)\bigl|_{X_{|\Sys|}}$, a contradiction.
		
		The other inclusion is in general not true.\qedhere
	\end{enumerate}
\end{proof}

%

Note that in the special case that all subsystems of the AiO problem \eqref{eq:Pwithindependentsubsystems} are single objective optimization problems, Proposition \ref{prop:Pwithindependentsubsets} implies that the ideal point given by 
$(f_1(\Sup(1,1,\emptyset)),\dots,f_{|\bm{\mathcal{S}}|}(\Sup(|\bm{\mathcal{S}}|,|\bm{\mathcal{S}}|,\emptyset))$ would be feasible and thus the only element in the image of the efficient set $\f(E(P_I))$ in the objective space.

An AiO problem \eqref{eq:P} may be formulated in such a way that it immediately suggests a decomposition into independent subsystems. However, decompositions into independent subsystems may be hidden in a given problem formulation, even though they exist and probably lead to a better understanding of the AiO problem. 


\begin{example}\label{example:decopmandideal}
	To illustrate the implications  of using different decompositions 
	we consider an AiO problem that can naturally be associated with  two different decomposition graphs. 
	\begin{equation}\label{eq:diagonalexample}
	\begin{aligned}
	\min \;& (\f_{1}(\x_1) , \f_{2}(\x_2) , \f_{3}(\x_3) , \f_{4}(\x_3))\\
	\text{s.t.} \; &  \x\in X.
	\end{aligned}
	\end{equation}
	In some real-life context when Problem~\eqref{eq:diagonalexample} models, for example, a complex project with two decision makers, the complex system graph may consist of two subsystem nodes. The first decision maker is concerned about objectives $\f_{1}$ and $\f_{3}$ (and thus manipulates variables $\x_1$ and $\x_3$), while the other decision maker cares about objectives $\f_{2}$ and $\f_{4}$ (and thus deals with variables $\x_2$ and $\x_3$). In this context, $\x_3$ are global variables and the two subsystems are not independent. 
	The resulting decomposition of Problem~\eqref{eq:diagonalexample} into two subsystems such that
	the variables $\x_3$ are global variables used in both subsystems is given as follows: subsystem $s_1$ operates on the variables in $X_1\times X_3$ and aims at optimizing the objective vector $\f_I(\x_1,\x_3) := (\f_{1}(\x_1) , \f_{3}(\x_3))$ while subsystem $s_2$ operates on the variables in $X_2\times X_3$ and has objectives $\f_{II}(\x_2,\x_3) := (\f_{2}(\x_2) , \f_{4}(\x_3))$. This decomposition is given by
	{\small\begin{equation}\notag
	(s_1)\,\begin{aligned}
	\min \;& \f_I(\x_1,\x_3) = (\f_{1}(\x_1) , \f_{3}(\x_3)) \\
	\text{s.t.} \; &  \x_1\in X_1\\
	& \x_3 \in X_3
	\end{aligned}
	\hspace{1cm}
	(s_2)\,\begin{aligned}
	\min \;&  \f_{II}(\x_2,\x_3) = (\f_{2}(\x_2) , \f_{4}(\x_3)) \\
	\text{s.t.} \; &  \x_2 \in X_2\\
	& \x_3 \in X_3.
	\end{aligned}
	\end{equation}}
	Note that this decomposition does not yield independent subsystems, and in general we only have
	$wE(P_I) \supseteq \Sup(1,\{1,2\},\emptyset) \cup \Sup(2,\{1,2\},\emptyset)$, see Proposition~\ref{lemma:subsystem_efficieny} above.

	In another real-life context, the associated complex system graph may consist of three subsystem nodes and imply the alternative decomposition as follows:
	{\small\begin{equation}\notag
	(s_1)  \begin{aligned}
	\min \;& \f_I(\x_1) \!=\! \f_{1}(\x_1) \\
	\text{s.t.} \; &  \x_1\in X_1    
	\end{aligned}
	\hspace{0.2cm}
	(s_2)  \begin{aligned}
	\min \;&  \f_{II}(\x_2) \!=\! \f_{2}(\x_2)  \\
	\text{s.t.} \; &  \x_2 \in X_2
	\end{aligned}
	\hspace{0.2cm}
	(s_3)  \begin{aligned}
	\min \;&  \f_{III}(\x_3) \!=\! (\f_{3}(\x_3) , \f_{4}(\x_3) ) \\
	\text{s.t.} \; &  \x_3 \in X_3.
	\end{aligned}
	\end{equation}}
	In this case, Proposition \ref{prop:Pwithindependentsubsets} implies the stronger result that 
	$E(P_I) = \Sup(1,1,\emptyset)\bigl|_{X_1} \times$ $\Sup(2,2,\emptyset)\bigl|_{X_2} \times \Sup(3,3,\emptyset)\bigl|_{X_3}$.

\end{example}

%
%

Example \ref{example:decopmandideal} illustrates that a decomposition of a complex system into independent subsystems may be beneficial since it guarantees that the efficient set of the AiO problem can be obtained as the Cartesian product of the subsystem efficient sets according to Proposition \ref{prop:Pwithindependentsubsets}. In the context of an application such as splitting a complex project into subprojects, a decomposition into independent subsystems thus allows an independent action of the decision makers of all subprojects. On the other hand, if such a decomposition is not available or, for some reason, not practical, the preimages of  subsystem level ideal outcomes  are in general not feasible (i.e., Proposition \ref{prop:Pwithindependentsubsets} does not apply) and there is a need for conflict resolution among the subsystems. Concepts of consensus among subsystems and system compromise solutions are discussed in Section \ref{sec:compromise}.

\subsection{Hierarchical Algorithms}\label{sec:hierarchical}


\noindent In the following, we present hierarchical solution approaches for the computation of superior solutions of complex systems. 
These algorithms consecutively solve the optimization problems in a hierarchical order given by the decision maker. 
We assume that the subsystems are given in an appropriate order and are numbered likewise.

The first approach solves the subsystem optimization problems in a given order, where $(\bm{\mathcal{S}},\bm{\mathcal{C}})$-validity has to be fulfilled in every iteration. This approach is presented in Algorithm~\ref{alg:classlex} and returns a subset of the $(\bm{\mathcal{S}},\bm{\mathcal{S}},\bm{\mathcal{C}})$-superior solutions.

\begin{algorithm}[htb]\small
	\SetKw{Compute}{compute}
	\SetKw{Break}{break}
	\SetKwInOut{Input}{input}\SetKwInOut{Output}{output}
	\SetKwComment{command}{right mark}{left mark}
	
	\Input{a complex system graph \eqref{eq:PG}}
	
	\Output{a subset of the $(\bm{\mathcal{S}},\bm{\mathcal{S}},\bm{\mathcal{C}})$-superior solutions}
	\BlankLine
	$X^{\Sup}_1 := \Sup(1,\bm{\mathcal{S}},\bm{\mathcal{C}})$\\
	
	\For{$i\leftarrow 2$ \KwTo $|\bm{\mathcal{S}}|$}{
		$X^{\Sup}_i \leftarrow	\Sup(i, X^{\Sup}_{i-1}, \emptyset)$ 
	}	
	
	\Return $X^{\Sup}_{|\bm{\mathcal{S}}|}$
	\caption{Hierarchical Solution Algorithm with all feasibility and consistency constraints \label{alg:classlex}}
\end{algorithm}

\begin{remark}
	In Algorithm~\ref{alg:classlex} the following sequence of inclusions holds:
	\begin{equation*}
	X_{\bm{\mathcal{S}},\bm{\mathcal{C}}} \supseteq X^{\Sup}_1 \supseteq X^{\Sup}_2 \supseteq \ldots \supseteq X^{\Sup}_{|\bm{\mathcal{S}|}} 
	\end{equation*}
\end{remark}

\begin{theorem}
	If \(X^{Sup}_{|\bm{\mathcal{S}|}}\) is computed as described in Algorithm~\ref{alg:classlex}, then the following inclusion holds: 
	\(X^{\Sup}_{|\bm{\mathcal{S}}|} \subseteq \Sup(\bm{\mathcal{S}},\bm{\mathcal{S}},\bm{\mathcal{C}}).\)
\end{theorem}
\begin{proof}
	Let $X^{\Sup}_{|\bm{\mathcal{S}|}} \neq \emptyset$ and let $\x \in X^{\Sup}_{|\bm{\mathcal{S}|}}$.
	Assume $\x \notin \Sup(\bm{\mathcal{S}},\bm{\mathcal{S}},\bm{\mathcal{C}})$. Then it holds that
	\(\exists \x' \in X_{\bm{\mathcal{S}}, \bm{\mathcal{C}}}: \f_i(\x') \leqq \f_i(\x) \ \forall i \in \bm{\mathcal{S}} \text{ and } \exists j \in \bm{\mathcal{S}}: \f_j(\x') \leq \f_j(\x) .
	\)
	Since \(\x\in X^{\Sup}_i\) and \(\x'\) is at least as good as \(\x\) in all subsystems, it holds that \(\x'\in X^{\Sup}_i\) for all \(i\in\bm{\mathcal{S}}\). 
	This is a contradiction to \(\x \in\Sup(j, X^{\Sup}_{j-1},\emptyset)\).
	%
	%
	%
\end{proof}
The knowledge of the $(\bm{\mathcal{S}},\bm{\mathcal{C}})$-valid set is a strong assumption and cannot be assumed to be available in practice.
Thus, in the second approach, we weaken this assumption: We do not require system validity from the beginning, but add feasibility and linking constraints
imposed by the subsystems iteratively at the moment when they are considered in the algorithm. 

In detail, the algorithm is initialized by solving the first subproblem neglecting all linking constraints. The result is a set of superior solutions for the first subproblem. Then, the next subsystem is processed. 
In iteration \(i\), \(i\in\{2,\ldots,|\bm{\mathcal{S}}|\}\), the set of linking nodes that are adjacent to the subsystem node assigned to the current subsystem under consideration and all nodes assigned to previously treated subsystems is defined by
\[
C_i:=\Bigl\{j\in\bm{\mathcal{C}}  \colon\exists k\in\{1,\ldots,i-1\}  \colon \sigma_i,\sigma_k \in \pred(\kappa_j) \Bigr\} .
\]
Using the linking constraints given by $C_i$, the solutions of the $i$th subsystem take into account the linking constraints to all previously considered subsystems. The algorithm then iterates: new subsystems are solved taking into account constraints coming from all previously solved subsystems. 

Note that a drawback of this iterative approach is that there is no guarantee that the algorithm terminates with a feasible solution. 
Even after the first iteration, i.\,e., after the computation of \(X_{2}^{\Sup}\), it cannot be guaranteed that \(X_{2}^{\Sup}\) is a non-empty set.

\begin{algorithm}[!ht]\small
	\SetKw{Compute}{compute}
	\SetKw{Break}{break}
	\SetKwInOut{Input}{input}\SetKwInOut{Output}{output}
	\SetKwComment{command}{right mark}{left mark}
	
	\Input{a complex system graph \eqref{eq:PG}}
	
	\Output{a subset of the $(\bm{\mathcal{S}},\bm{\mathcal{S}},\bm{\mathcal{C}})$-superior solutions}
	
	\BlankLine
	$X_1^{\Sup}\leftarrow \Sup(1,X_{1,\emptyset},\emptyset)$\\
	
	\For{$i\leftarrow 2$ \KwTo $|\bm{\mathcal{S}}|$}{
		$X_{i}^{\Sup} \leftarrow	\Sup(i,X_{i,\emptyset}\cap X_{i-1}^{\Sup},C_i)$
	}	
	
	\Return $X_{\bm{\mathcal{|S|}}}^{\Sup}$
	\caption{Hierarchical Solution Algorithm with step-by-step inclusion of feasibility and consistency  \label{alg:HierarchicalAlgorithm}}
\end{algorithm}

%
%
%
%
%

\begin{theorem}
	If \(X^{Sup}_{|\bm{\mathcal{S}|}}\) is computed as described in Algorithm~\ref{alg:HierarchicalAlgorithm}, then the following inclusion holds: 
	\([X^{\Sup}_{|\bm{\mathcal{S}}|} \subseteq \Sup(\bm{\mathcal{S}},\bm{\mathcal{S}},\bm{\mathcal{C}}).\)
\end{theorem}
\begin{proof}
	Let $\x \in X^{\Sup}_{|\bm{\mathcal{S}}|}$. 
	Assume $\x \notin \Sup(\bm{\mathcal{S}},\bm{\mathcal{S}},\bm{\mathcal{C}})$. Then it holds that
	\(
	\exists \x' \in X_{\bm{\mathcal{S}}, \bm{\mathcal{C}}}: \f_i(\x') \leqq \f_i(\x) \ \forall i \in \bm{\mathcal{S}} \text{ and } \exists j \in \bm{\mathcal{S}}: \f_j(\x') \leq \f_j(\x).
	\)
	Since \(\x\in X^{\Sup}_i\) and \(\f_i(\x')\leqq \f_i(x)\) and \(\x'\in X_{\bm{\mathcal{S}}, \bm{\mathcal{C}}}\), it holds that \(\x'\in X^{\Sup}_i\) for all \(i\in\bm{\mathcal{S}}\). 
	This is a contradiction to \(\x \in\Sup(j, X_{j,\emptyset} \cap X^{\Sup}_{j-1},C_j)\), since \(\f_j(\x') \leq \f_j(\x)\).
	%
	%
\end{proof}

%
%
%


Since in Algorithm 2 feasibility may not be achieved, in the following we relax the superiority to $\varepsilon$-superiority to increase the chances for attaining it.

\begin{definition}
	Given \(\varepsilon\geq 0\) an \((S,C)\)-valid  solution \(\x\in X_{S,C}\) is called \emph{\((F,S,C)\)-\(\varepsilon\)-superior} if 
	\[\nexists \bar{\x}\in X_{S,C}: (1+\varepsilon)\, \f(\bar{\x})\preceq_{(F,S,C)} \f(\x) .\] 
	The set of all $(F,S,C)$-\(\varepsilon\)-superior solutions is denoted by \(\epsSup(F,S,C)\).
\end{definition}

Note that weakly and strictly $\varepsilon$-superior solutions can be defined analogously, and that for \(\varepsilon=0\) the definition of \(\varepsilon\)-superiority and superiority are equivalent. 

\begin{algorithm}[!ht]\small
	\SetKw{Compute}{compute}
	\SetKw{Break}{break}
	\SetKwInOut{Input}{input}\SetKwInOut{Output}{output}
	\SetKwComment{command}{right mark}{left mark}
	
	\Input{a complex system graph \eqref{eq:PG}, a parameter $\varepsilon >0$}
	
	\Output{a subset of the $(\bm{\mathcal{S}},\bm{\mathcal{S}},\bm{\mathcal{C}})$-\(\varepsilon\)-superior solutions}
	
	$X_{1}^{\epsSup} \leftarrow	\epsSup\bigl(1,X_{1,\emptyset} ,\emptyset \bigr)$\\
	
	\For{$i\leftarrow 2$ \KwTo $|\bm{\mathcal{S}}|-1$}{
		$X_{i}^{\epsSup} \leftarrow	\epsSup\bigl(i,X_{i,\emptyset} \cap X_{i-1}^{\epsSup} ,C_i \bigr)$
	}	
	$X_{|\bm{\mathcal{S}}|}^{\epsSup} \leftarrow	\Sup\bigl(|\bm{\mathcal{S}}|,X_{|\bm{\mathcal{S}}|,\emptyset} \cap X_{|\bm{\mathcal{S}}|-1}^{\epsSup} , C_{|\bm{\mathcal{S}}|} \bigr)$\\
	\Return $X_{|\bm{\mathcal{S}}|}^{\epsSup}$
	\caption{Hierarchical \(\varepsilon\)-superior Solution Algorithm \label{alg:HierarchicalEpsAlgorithm}}
\end{algorithm}

Algorithm \ref{alg:HierarchicalEpsAlgorithm} is similar to Algorithm \ref{alg:HierarchicalAlgorithm} with the exception that we compute the set of \(\varepsilon\)-superior solutions in each iteration. However, in the last iteration there is no need to take care of upcoming feasibility issues such that it is sufficient to compute the set of superior solutions.

\begin{theorem}
	If \(X^{\epsSup}_{|\bm{\mathcal{S}|}}\) is computed as described in Algorithm~\ref{alg:HierarchicalEpsAlgorithm}, then the following inclusion holds: 
	\(X^{\epsSup}_{|\bm{\mathcal{S}}|} \subseteq \epsSup(\bm{\mathcal{S}},\bm{\mathcal{S}},\bm{\mathcal{C}}).\)
\end{theorem}
\begin{proof}
	Let $\x \in X^{\epsSup}_{|\bm{\mathcal{S}}|}$. 
	Assume $\x \notin \epsSup(\bm{\mathcal{S}},\bm{\mathcal{S}},\bm{\mathcal{C}})$. Then it holds that
	\(
	\exists \x' \in X_{\bm{\mathcal{S}}, \bm{\mathcal{C}}}: (1+\varepsilon) \f_i(\x') \leqq \f_i(\x) \ \forall i \in \bm{\mathcal{S}} \text{ and } \exists j \in \bm{\mathcal{S}}: (1+\varepsilon)\f_j(\x') \leq \f_j(\x).
	\)
	Since \(\x\in X^{\epsSup}_i\) and \((1+\varepsilon)\f_i(\x')\leqq \f_i(x)\) and \(\x'\in X_{\bm{\mathcal{S}}, \bm{\mathcal{C}}}\), it holds that \(\x'\in X^{\epsSup}_i\) for all \(i\in\bm{\mathcal{S}}\). 
	This is a contradiction to \(\x \in\epsSup(j, X_{j,\emptyset} \cap X^{\epsSup}_{j-1},C_j)\), since \((1+\varepsilon)\f_j(\x') \leq \f_j(\x)\).
\end{proof}

Both Algorithms~\ref{alg:HierarchicalEpsAlgorithm} and \ref{alg:HierarchicalEpsAlgorithmAdaptive} implement the same solution strategy. Algorithm~\ref{alg:HierarchicalEpsAlgorithmAdaptive} additionally includes an adaptive selection of the approximation parameter~\(\varepsilon\). Thereby the minimum value of \(\varepsilon\) (up to an accuracy of \(\delta\)) is selected such that set of computed \(\varepsilon\)-superior solutions is non-empty. In detail, the algorithm starts with \(\varepsilon=0\). If no feasible solution is found, the lower bound \(lb\) is set to \(\varepsilon\) and a new larger \(\varepsilon\) is chosen. On the other hand, if a solution is found, the upper bound \(ub\) is set to \(\varepsilon\) and a smaller \(\varepsilon\) is chosen. In the algorithm, \(\varepsilon\) is chosen via bisection search. Note that there are other possible selection methods for \(\varepsilon_j\in(lb,ub)\), like e.\,g., the golden section search. 


\begin{algorithm}[!ht]\small
	\SetKw{Compute}{compute}
	\SetKw{Break}{break}
	\SetKwInOut{Input}{input}\SetKwInOut{Output}{output}
	\SetKwComment{command}{right mark}{left mark}
	
	\Input{a complex system graph \eqref{eq:PG}, $M>\delta>0, M$ large}
	
	\Output{a guaranteed accuracy $\varepsilon$ and a subset of the $(\bm{\mathcal{S}},\bm{\mathcal{S}},\bm{\mathcal{C}})$-\(\varepsilon\)-superior solutions}
	
	$j \leftarrow 1$  \\
	$lb \leftarrow 0$ \\
	$ub \leftarrow M$ \\
	\(\varepsilon_1 \leftarrow 0\) \\
	\While{$ub - lb >\delta$}{
		$X_{1}^{\epsjSup} \leftarrow	\epsjSup\bigl(1,X_{1,\emptyset} ,\emptyset \bigr)$\\
		
		$i \leftarrow 1$ \\

		\While{$X_i^{\epsjSup} \neq \emptyset \land i < |\bm{\mathcal{S}}|$}
		{
			$i \leftarrow i+1$ \\
			$X_i^{\epsjSup} \leftarrow \epsjSup(i,X_{i,\emptyset}\cap X_{i-1}^{\epsjSup},C_i)$	
			
		}
		\If{$i=|\bm{\mathcal{S}}|\land X_i^{\epsjSup}\neq\emptyset$}
		{
			$ub \leftarrow \varepsilon_j$ \\
			$\varepsilon_{j+1} \leftarrow (lb+ub)/2$ \\
			$X_{|\bm{\mathcal{S}}|}^{\varepsilon^*\text{-}\Sup} \leftarrow X_{|\bm{\mathcal{S}}|}^{\epsjSup} $ 
		}
		\Else{
			$lb \leftarrow \varepsilon_j$ \\
			$\varepsilon_{j+1}\leftarrow(lb+ub)/2$ \\
		}
		$j \leftarrow j+1$
		
	}	
	
	\Return  $\varepsilon^* \leftarrow ub$, \(X_{|\bm{\mathcal{S}}|}^{\varepsilon^*\text{-}\Sup}\)
	\caption{Hierarchical \(\varepsilon\)-superior Solution Algorithm with $\varepsilon$-updates \label{alg:HierarchicalEpsAlgorithmAdaptive}}
\end{algorithm}

\subsection{Scalarization Based Approach}
\noindent Since there are several ways to decompose a complex system, the specific structure might be important. For illustration, we consider a special case of a complex system with additive separable objective functions and separable constraints.

We define the variable vectors ${\x}_{s_i}=(\x_0,\x_i)$, the feasible sets
$X_i=\{\x_{s_i}\in\mathbb{R}^{|\pred(\sigma_i)|} : A_0 \, \x_0 \leqq \b_0, \, A_i \, \x_i \leqq \b_i\}$, and the objective function vectors $\f_i(\x_{s_i}) = (\g_i(\x_i) + \g_{0i}(\x_0) )$, \(i = 1,\ldots,|\bm{\mathcal{S}}|\).

Consider the following AiO problem.
\begin{subequations}        
	\begin{align}
	\min \;&\Bigl(\g_1(\x_1) + \g_{01}(\x_0),\ldots,\g_{|\bm{\mathcal{S}}|}(\x_{|\bm{\mathcal{S}}|}) + \g_{0{|\bm{\mathcal{S}}|}}(\x_0) \Bigr)\\
	\text{s.t.} \;& A_i \, \x_i \leqq \b_i \ ,i=0,\ldots,{|\bm{\mathcal{S}}|}  \label{eqa:sc1}\\
	& \x_i \geqq \0 \ ,i=0,\ldots,{|\bm{\mathcal{S}}|}
	\end{align}
\end{subequations}

First, we scalarize the AiO problem and then we decompose it in the following way for \(i=1,\ldots, |\bm{\mathcal{S}}|\):

\begin{equation}\notag
\begin{aligned}
\min \;& \bm{w}_i^\top \g_i(\x_i)\\
\text{s.t.} \;& A_i \, \x_i \leqq \b_i\\
& \x_i \geqq \0
\end{aligned}\quad (S_i)
\hspace{2cm}
\begin{aligned}
\min \;& (\bm{w}_1^\top,\ldots,\bm{w}_{|\bm{\mathcal{S}}|}^\top) \left(\begin{array}{c}\g_{01}(\x_0)\\ \vdots \\ \g_{0{|\bm{\mathcal{S}}|}}(\x_0)\end{array}\right)\\
\text{s.t.} \;& A_0 \, \x_0 \leqq \b_0 \\
& \x_0 \geqq \0
\end{aligned}\quad (S_0)
\end{equation}

\begin{proposition}
	Let \(\x_i^*\) be optimal for \((S_i)\), \(i = 0,\ldots,|\bm{\mathcal{S}}|\) and let \(\x^* \coloneqq (x_0^*,\ldots,x_{|\bm{\mathcal{S}|}}^*)\). Then, \(\x^*\) is weakly efficient for the AiO problem.
\end{proposition}
\begin{proof}
	Consider the weighted sum of the objective function:
	\begin{equation}\notag
	\operatorname{min}  \sum_{i=1}^{|\bm{\mathcal{S}}|}\bm{w}_i^\top(\g_i(\x_i)+ \g_{0i}(\x_0))
	= \operatorname{min} \sum_{i=1}^{|\bm{\mathcal{S}}|} \bm{w}_i^\top \g_i(\x_i) + \sum_{i=1}^{|\bm{\mathcal{S}}|} \bm{w}_i^\top \g_{0i}(\x_0) 
	\end{equation}
	This is equivalent to solving the proposed decomposition into the problems \((S_0)\) and \((S_i)\), \(i = 1,\ldots,|\bm{\mathcal{S}}|\) due to the separability of the constraints.
	
	Further, we know that optimal solutions of the weighted sum of the AiO problem are at least weakly efficient for it, which concludes the proof.
\end{proof}

\section{Obtaining a Compromise}\label{sec:compromise}

\noindent For complex systems with a large number of subsystems, objective functions and linking constraints, finding superior solutions is highly challenging. In the case that the algorithms presented in Section \ref{sec:hierarchical} do not return any satisfying solutions, it is thus a viable option to strive for a \emph{compromise} between all subsystems. There are many possible ways to define compromise solutions for complex systems. For example, we may compromise on subsystem feasibility and/or system validity by relaxing one or several constraints, or we may trade-off between objective functions by ignoring some of them. We will exemplify the concept of compromise solutions for complex systems on a distance-based model that is motivated by compromise programming in multiobjective optimization \citep[see, e.g.,][]{RJ92}. In this context, a compromise solution is a solution which is not necessarily superior for every subsystem, but which is as close as possible to the superior sets of all subsystems with 
respect to some distance measure.

The concept of distance-based compromise can be formulated either in the decision or in the objective space. We refer to the former as \emph{decision space compromise} and to the latter as \emph{objective space compromise}. A major challenge when considering objective space compromise is that the preimage of a selected outcome vector is usually not unique. Moreover, preimages are generally hard to compute, depending on the type of the objectives.

When considering decision space compromise, the computational complexity depends (besides the complexity of finding system valid solutions) on the chosen distance function and on the objective that is used to evaluate candidate solutions. The problem can be interpreted as a specific location problem that is formulated on the system valid set $X_{\bm{\mathcal{S}},\bm{\mathcal{C}}}\subseteq\R^{|\bm{\mathcal{V}}|}$: Given, for example, the superior sets of all subsystems, find a system valid solution that minimizes (1) the (weighted) sum of distances from these sets (such problems are referred to as Weber or median problems), or (2) the maximum of all (weighted) distances from these sets. 



\subsection{Decision Space Compromise Based on Median Problems}
\label{sec:compromise:basics}

\noindent While it is generally difficult to compute $(\bm{\mathcal{S}},\bm{\mathcal{S}},\bm{\mathcal{C}})$-superior solutions, finding $(i,i,\emptyset)$-superior solutions, $i=1,\dots,|\bm{\mathcal{S}}|$, is often considerably easier since it requires only ``local'' knowledge from the individual subproblems. We thus use the sets $\Sup(i,i,\emptyset)$ as reference sets when searching for compromise solutions. 
We consider classical $l_p$-distances for measuring distances and focus, in particular, on the case $p=1$. The $l_1$-distance between a point $\x\in\mathbb{R}^{|\bm{\mathcal{V}}|}$ and a closed set $M\subset\mathbb{R}^{|\bm{\mathcal{V}}|}$ is defined as 
$ l_1(\x,M):=\min_{\bar{\x}\in M} l_1\bigl(\x,\bar{\x}\bigr)$, 
where $l_1(\x,\bar{\x})=\sum_{i=1}^{|\bm{\mathcal{V}}|} |\bar{x}_i-x_i|$ denotes the $l_1$-distance of $\x,\bar{\x}\in\mathbb{R}^{|\bm{\mathcal{V}}|}$.

\begin{definition}[$l_1$-median compromise]\mbox{}
	Suppose that the subsystem superior sets $\Sup(i,i,\emptyset)$ are given for all $i\in\bm{\mathcal{S}}$, and consider the following median location problem with $l_1$ distances:
	\begin{equation}\tag{$C_W(\Sup)$}\label{eq:Weber_Original}
	\begin{aligned}
	\min \;& \sum_{i=1}^{|\bm{\mathcal{S}}|}l_1(\x,\Sup(i,i,\emptyset))\\
	\text{s.t.} \; &  \x\in\mathbb{R}^{|\bm{\mathcal{V}}|}.
	\end{aligned}
	\end{equation}
	Then an optimal solution of \eqref{eq:Weber_Original} is called \emph{$l_1$-median-compromise}.
\end{definition}
Note that in general an optimal solution of problem \eqref{eq:Weber_Original} may not be system valid.
An additional difficulty arises from the computation of the distance between a single point (the sought compromise solution) and a reference set $\Sup(i,i,\emptyset)$, $i\in\bm{\mathcal{S}}$. This requires in general knowledge of the entire set $\Sup(i,i,\emptyset)$. 
Moreover, (\ref{eq:Weber_Original}) is usually still a highly complex nonconvex optimization problem since $\Sup(i,i,\emptyset)$ is a nonconvex set in general. The goal of reducing the computational effort is thus not reached with this approach.
As a consequence, it may be advantageous to represent the subsystem superior sets $\Sup(i,i,\emptyset)$, $i\in\bm{\mathcal{S}}$, in problem~(\ref{eq:Weber_Original}) by appropriately chosen (sets of) reference points or reference sets. 

\begin{definition}[Reference set]\mbox{} \label{def:referenceset}
	Let $i\in\bm{\mathcal{S}}$. A finite set $R_i:=\{\bm{r}^1,\dots,\bm{r}^{|R_i|}\}\subseteq\Sup(i,i,\emptyset)\subseteq\mathbb{R}^{|\bm{\mathcal{V}}|}$ of representative points for $\Sup(i,i,\emptyset)$ is called a set of \emph{reference points} for subsystem $s_i$. The union of all reference points over all subsystems, $\bm{\mathcal{R}}=\bigcup_{i\in\bm{\mathcal{S}}} R_i$, is called \emph{reference set}. 
\end{definition}

\begin{definition}[$l_1$-median compromise based on reference points]\mbox{}
	Let $\bm{\mathcal{R}}$ be a reference set. Then the \emph{$l_1$-median problem based on reference points} can be stated as
	\begin{equation}\tag{$C_W(\bm{\mathcal{R}})$}\label{eq:Weber_furtherrelaxed}
	\begin{aligned}
	\min \;& \sum_{j=1}^{|\bm{\mathcal{R}}|}l_1(\x,\bm{r}^j)\\
	\text{s.t.} \; &  \x\in\mathbb{R}^{|\bm{\mathcal{V}}|}.
	\end{aligned}
	\end{equation}
\end{definition}
Problem (\ref{eq:Weber_furtherrelaxed}) is a well-known median problem with $l_1$-norms in $\mathbb{R}^{|\bm{\mathcal{V}}|}$, see, for example, \cite{Sim50} for an early reference. 

While unconstrained median location problems are well-understood and relatively easy to solve when formulated in the plane $\R^2$, this is generally not the case in higher dimensions. However, there also exist specific solution methods for median problems in arbitrary dimension. 
The following result is a classical result in location theory and particularly useful when all reference points are system valid and the system valid set is convex.
\begin{lemma}\label{OptimalSolutionInConvexHull}
	Let a location problem of the form of ~(\ref{eq:Weber_furtherrelaxed}) be given. Then there exists an optimal solution $\x^*$ of~(\ref{eq:Weber_furtherrelaxed}) such that $\x^*$ is in the convex hull of $\bm{\mathcal{R}}$.
\end{lemma}
\begin{proof}
	See \cite{RD94}.
\end{proof}
\begin{remark}
	If all points in the reference set $\bm{\mathcal{R}}$ are system valid and $X_{\bm{\mathcal{S}},\bm{\mathcal{C}}}$ is convex, Lemma~\ref{OptimalSolutionInConvexHull} ensures that there is an optimal solution of~(\ref{eq:Weber_furtherrelaxed}) that is system valid as well. However, if some of the reference points are only subsystem valid (and not system valid), this does not hold in general. This is illustrated by a small example in Figure~\ref{fig:SystemvalidCounterexample}.
\end{remark}

\begin{figure}
	\centering\small
\begin{tikzpicture}[scale=0.6]
\draw  (3,4) -- (6,2) -- (3,0) -- (0,2);
\draw  (7,4) -- (4,2) -- (7,0) -- (10,2);
\draw [thick,dash pattern={on 7pt off 2pt on 1pt off 3pt}] (0,2) -- (3,4);
\draw [thick, dashed] (10,2) -- (7,4);
\draw [thick, dotted] (3,4) -- (7,4);
\node [below] at (3,2.35) {$X_{1,\bm{\mathcal{C}}}$};
\node [below] at (7,2.35) {$X_{2,\bm{\mathcal{C}}}$};
\node [below] at (5,2.35) {$X_{\bm{\mathcal{S}},\bm{\mathcal{C}}}$};
\draw [fill=black] (3,4) circle (.4ex);
\node [below=-4pt] at (3,4.6) {$\bm{r}^1$};
\draw [fill=black] (7,4) circle (.4ex);
\node [below=-4pt] at (7,4.6) {$\bm{r}^2$};
\draw [fill=black] (4,4) circle (.4ex);
\node [below=-4pt] at (4,4.6) {$\bm{x}^*$};
\node [left] at (1.35,3) {$\Sup(1,1,\bm{\mathcal{C}})$};
\node [right] at (8.7,3) {$\Sup(2,2,\bm{\mathcal{C}})$};
\end{tikzpicture}
	\caption{\label{fig:SystemvalidCounterexample} Subsystem and system valid sets for a complex system in $\R^2$ decomposed into two subsystems $s_1$ and $s_2$. If the reference points $\bm{r}^1$ and $\bm{r}^2$ are chosen as depicted, then an optimal solution of problem \eqref{eq:Weber_furtherrelaxed} must lie on the dotted line, which does not intersect the system valid set.}
\end{figure}
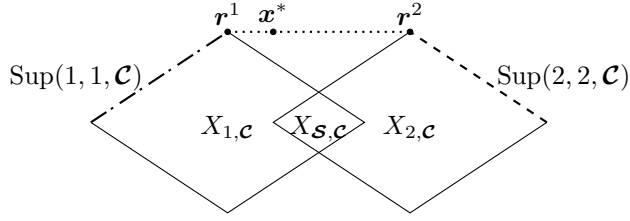

\subsection{Computing Compromise Solutions}

\noindent In the following we will adapt an algorithm proposed by \cite{RD94} to solve~(\ref{eq:Weber_furtherrelaxed}). It guarantees an optimal solution contained in the convex hull of the set of reference points.
Given a reference set $\bm{\mathcal{R}}$, a compromise solution $\bm{x}^*\in\R^{|\bm{\mathcal{V}}|}$ can be computed by solving 
\begin{align*}\min_{\x\in\mathbb{R}^{|\bm{\mathcal{V}}|}}\sum_{\bm{r}^{j}\in\bm{\mathcal{R}}} l_1(\bm{x},\bm{r}^{j}) & :=\min_{\x\in\mathbb{R}^{|\bm{\mathcal{V}}|}}\sum_{\bm{r}^{j}\in\bm{\mathcal{R}}}\sum_{k=1}^{|\bm{\mathcal{V}}|} |x_k-r_k^{j}| 
 =\sum_{k=1}^{|\bm{\mathcal{V}}|}\ \min_{\x\in\mathbb{R}^{|\bm{\mathcal{V}}|}}\sum_{\bm{r}^{j}\in\bm{\mathcal{R}}}\ \  |x_k-r_k^{j}|.
\end{align*}
Thus, the components of $\x$ can be optimized independently. We define
\begin{align*}x_k^*:= \argmin_{x_k\in\mathbb{R}}\sum_{\bm{r}^{j}\in\bm{\mathcal{R}}}\ \  |x_k-r_k^{j}| \quad\text{for all } k\in\{1,\dots ,|\bm{\mathcal{V}}|\}.
\end{align*}

\begin{lemma}\label{lemma:l1_optimality}
	Let $\bm{\mathcal{R}}$ be a reference set. If for $\x^*:=(x^*_1,\ldots,x^*_{|\bm{\mathcal{V}}|})^T$ it holds that
	\begin{align*}|\{\bm{r}^{j}\in \bm{\mathcal{R}}:{r}_k^{j}\leq x_k^*\} | &\geq |\{\bm{r}^{j}\in \bm{\mathcal{R}}:{r}_k^{j}> x_k^*\} |,\ k=1,\ldots,|\bm{\mathcal{V}}|, \;\text{and}\\
	|\{\bm{r}^{j}\in \bm{\mathcal{R}}:{r}_k^{j}\geq x_k^*\} | &\geq |\{\bm{r}^{j}\in \bm{\mathcal{R}}:{r}_k^{j}< x_k^*\} |,\ k=1,\ldots,|\bm{\mathcal{V}}|,\end{align*}
	then $\bm{x}^*$ is optimal for~(\ref{eq:Weber_furtherrelaxed}).
\end{lemma}

\begin{figure}
	\centering
	\begin{tikzpicture}[scale=1]\small
	\draw [->] (0,0) -- (7,0);
	\draw (3,-0.1) -- (3.1,0.1) -- (3.2,-0.1) -- (3.3,0.1) -- (3.4,-0.1) -- (3.5,0.1) -- (3.6,-0.1) -- (3.7,0.1) -- (3.8,-0.1) -- (3.9,0.1) -- (4,-0.1) -- (4.1,0.1) -- (4.2,-0.1) -- (4.3,0.1) -- (4.4,-0.1) -- (4.5,0.1) -- (4.6,-0.1) -- (4.7,0.1) -- (4.8,-0.1) -- (4.9,0.1) -- (5,-0.1);
	\draw (1,0.08) -- (1,-0.08);
	\draw (1.4,0.08) -- (1.4,-0.08);
	\draw (3,0.08) -- (3,-0.15);
	\draw (5,0.08) -- (5,-0.15);
	\draw (5.5,0.08) -- (5.5,-0.08);
	\draw (6,0.08) -- (6,-0.08);
	\node [below] at (1,-0.08) {$\bm{r}^1$};
	\node [below] at (1.4,-0.08) {$\bm{r}^2$};
	\node [below] at (3,-0.15) {$\bm{r}^3$};
	\node [below] at (5,-0.15) {$\bm{r}^4$};
	\node [below] at (5.5,-0.08) {$\bm{r}^5$};
	\node [below] at (6,-0.08) {$\bm{r}^6$};
	\end{tikzpicture}
	\caption{\label{fig:l1onedim} Illustration of the set of optimal solutions (jagged lines) of problem~(\ref{eq:Weber_furtherrelaxed}) in one dimension.}
\end{figure}
\begin{proof}
	See \cite{Jue80}. 
\end{proof}
\begin{example} Figure~\ref{fig:l1onedim} illustrates a  simple one dimensional example. 
	If $\bm{x}^*$ is not between $\bm{r}^3$ and $\bm{r}^4$, it is not optimal. 
	Note that $\x^*$ is in general not unique.
\end{example}


Lemma \ref{lemma:l1_optimality} implies a simple and efficient algorithm for \eqref{eq:Weber_furtherrelaxed} that is based on the independent solution of $|\bm{\mathcal{V}}|$ subproblems. Since all considered distances are unweighted, these subproblems reduce to a simple counting problem in order to find the ``midpoint'' of the reference points in the considered component. 
The set of optimal solutions of \eqref{eq:Weber_furtherrelaxed} defines a hyperrectangle in $\R^{|\bm{\mathcal{V}}|}$ that has a nonempty intersection with the convex hull of the reference set. If the system valid set is convex, then the intersection with the system valid set is also non-empty. Algorithm~\ref{alg:l1Algorithm} thus first determines the hyperrectangle containing the optimal solutions of \eqref{eq:Weber_furtherrelaxed} (Steps 1-4) and then solves a simple linear system to determine one optimal solution that is within the convex hull of the reference set (Step 5).

\begin{algorithm}[!ht]\small
	\SetKw{Compute}{compute}
	\SetKw{Break}{break}
	\SetKwInOut{Input}{input}\SetKwInOut{Output}{output}
	\SetKwComment{command}{right mark}{left mark}
	
	\Input{a reference set $\bm{\mathcal{R}}$}
	
	\Output{a compromise solution $\x^*$}
	
	\BlankLine
	Set $l:=\left\lceil\frac{|\bm{\mathcal{R}}|}{2}\right\rceil$ and $u:=\left\lfloor\frac{|\bm{\mathcal{R}}|}{2}\right\rfloor+1$\\ 
	
	\For{$k=1,\dots ,|\bm{\mathcal{V}}|$ }{
		Sort $r_k^{j}$, $j=1,\dots,|\bm{\mathcal{R}}|$, such that $r^{\pi (1)}_k \leq\dots\leq r^{\pi (|\bm{\mathcal{R}}|)}_k$\\
		Set ${lb}_k:=r^{\pi (l)}_k$ and ${ub}_k:=r^{\pi (u)}_k$\\
	}	
	Solve the following linear system: 
	\begin{alignat*}{3}
	  \sum_{j=1}^{|\bm{\mathcal{R}}|}\lambda_{j}\bm{r}^{j} & =\x^*, &\qquad && {lb}_k\leq x_k^* & \leq {ub}_k \quad  \forall k=1,\dots,|\bm{\mathcal{V}}|,\\
	  \sum_{j=1}^{|\bm{\mathcal{R}}|}\lambda_{j} & =1, &&& \lambda_{j} & \geq 0 \quad \forall j=1,\dots,|\bm{\mathcal{R}}|
	\end{alignat*}
	\Return $\bm{x}^*$
	\caption{$l_1$-Compromise Algorithm \label{alg:l1Algorithm}}
\end{algorithm}


\begin{example}\label{example:l1Algo}
	The following example is used to illustrate Algorithm~\ref{alg:l1Algorithm}, see also Figure~\ref{fig:l1Algo}.
	{\small
	\begin{equation}\notag
	(s_1)\,\begin{aligned}
	\min \;& f_1(x_1,x_2) = x_1+x_2 \\
	\text{s.t.} \; & 0\leq x_1,\, x_2 \leq 2\\
	& \frac{1}{2}\leq x_1+x_2\\
	\end{aligned}
	\hspace{0.3cm}
	(s_2)\,\begin{aligned}
	\min \;& f_2(x'_1,x'_2) = \frac{2}{3}x'_1-x'_2 \\
	\text{s.t.} \; &   \frac{2}{3}x'_1-x'_2 \leq 1\\
	&   \frac{3}{2}x'_1-x'_2 \leq \frac{5}{2}\\
	\end{aligned}
	\hspace{0.3cm}
	(s_3)\,\begin{aligned}
	\min \;& f_3(x''_1) = -x''_1 \\
	\text{s.t.} \; &  0\leq x''_1 \leq 3\\
	\end{aligned}
	\end{equation}}
	Additionally, there are the following easy-linking constraints: $c_1(x_1,\, x'_1)=x_1-x'_1=0$, $c_2(x_1,\, x''_1)=x_1-x''_1=0$, $c_3(x'_1,\, x''_1)=x'_1-x'_1=0$ and $c_4(x_2,\, x'_2)=x_2-x'_2=0$. 
	The points $\bm{r}^{1}=(0.5,0)^T$, $\bm{r}^{2}=(0.5,0)^T$, $\bm{r}^{3}=(1.5,0)^T$, $\bm{r}^{4}=(1.8,0.2)^T$ and $\bm{r}^{5}=(2,0.5)^T$, $\bm{r}^{6}=(2,0.5)^T$ are used as reference points in Algorithm~\ref{alg:l1Algorithm}. The gray rectangle depicts the set of optimal solutions of \ref{eq:Weber_furtherrelaxed} using these reference points. The dashed area are those optimal solutions of \ref{eq:Weber_furtherrelaxed} that are elements of the convex hull of the reference set as well. One element of this set is the solution of Algorithm~\ref{alg:l1Algorithm}.
\end{example}

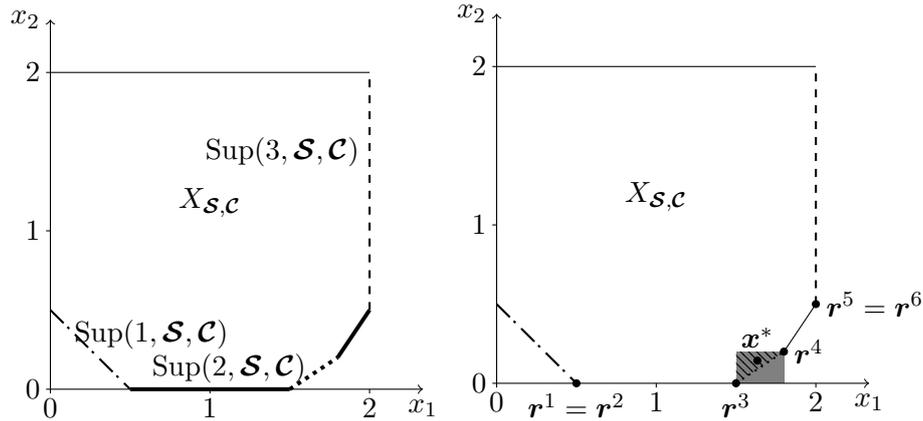
\begin{figure}
	\centering
	\begin{tikzpicture}[scale=0.7]
	\draw [->] (0,0) -- (7,0);
	\draw [->] (0,0) -- (0,7);
	\draw (0,0) -- (0,-0.08);
	\node [below] at (0,0) {$0$};
	\draw (3,0) -- (3,-0.08);
	\node [below] at (3,0) {$1$};
	\draw (6,0) -- (6,-0.08);
	\node [below] at (6,0) {$2$};
	\node [below] at (7,0) {$x_1$};
	\draw (0,0) -- (-0.08,0);
	\node [left] at (0,0) {$0$};
	\draw (0,3) -- (-0.08,3);
	\node [left] at (0,3) {$1$};
	\draw (0,6) -- (-0.08,6);
	\node [left] at (0,6) {$2$};
	\node [left] at (0,7) {$x_2$};
	\draw  [line width=1.5pt] (5.4,0.6) -- (6,1.5);
	\draw [line width=1.5pt, dotted] (4.5,0) -- (5.4,0.6);
	\draw [line width=1.5pt] (1.5,0) -- (4.5,0);
	\draw (6,6) -- (0,6);
	\draw [thick,dash pattern={on 7pt off 2pt on 1pt off 3pt}] (0,1.5) -- (1.5,0);
	\draw [thick, dashed] (6,1.5) -- (6,6);
	\node [below] at (3,4) {$X_{\bm{\mathcal{S}},\bm{\mathcal{C}}}$};
	\node [below] at (1.9,1.5) {$\Sup(1,\bm{\mathcal{S}},\bm{\mathcal{C}})$};
	\node [left] at (5,0.4) {$\Sup(2,\bm{\mathcal{S}},\bm{\mathcal{C}})$};
	\node [left] at (6,4.5) {$\Sup(3,\bm{\mathcal{S}},\bm{\mathcal{C}})$};
	\end{tikzpicture}\begin{tikzpicture}[scale=0.7]
	\fill[color=gray] (4.5,0) rectangle (5.4,0.6);
	\draw [->] (0,0) -- (7,0);
	\draw [->] (0,0) -- (0,7);
	\draw (0,0) -- (0,-0.08);
	\node [below] at (0,0) {$0$};
	\draw (3,0) -- (3,-0.08);
	\node [below] at (3,0) {$1$};
	\draw (6,0) -- (6,-0.08);
	\node [below] at (6,0) {$2$};
	\node [below] at (7,0) {$x_1$};
	\draw (0,0) -- (-0.08,0);
	\node [left] at (0,0) {$0$};
	\draw (0,3) -- (-0.08,3);
	\node [left] at (0,3) {$1$};
	\draw (0,6) -- (-0.08,6);
	\node [left] at (0,6) {$2$};
	\node [left] at (0,7) {$x_2$};
	\fill[pattern=north west lines, pattern color=black] (4.5,0) -- (5.4,0.6) -- (4.5,0.6) -- (4.5,0);
	\draw  (5.4,0.6) -- (6,1.5);
	\draw [thick, dotted] (4.5,0) -- (5.4,0.6);
	\draw (6,6) -- (0,6);
	\draw [thick,dash pattern={on 7pt off 2pt on 1pt off 3pt}] (0,1.5) -- (1.5,0);
	\draw [thick, dashed] (6,1.5) -- (6,6);
	\node [below] at (3,4) {$X_{\bm{\mathcal{S}},\bm{\mathcal{C}}}$};
	\draw [fill=black] (1.5,0) circle (.4ex);
	\node [below] at (1.5,0) {$\bm{r}^{1}=\bm{r}^{2}$};
	\draw [fill=black] (6,1.5) circle (.4ex);
	\node [right] at (6,1.5) {$\bm{r}^{5}=\bm{r}^{6}$};
	\draw [fill=black] (4.5,0) circle (.4ex);
	\node [below] at (4.5,0) {$\bm{r}^{3}$};
	\draw [fill=black] (5.4,0.6) circle (.4ex);
	\node [right] at (5.4,0.6) {$\bm{r}^{4}$};
	\draw [fill=black] (4.9,0.43) circle (.4ex);
	\node [above] at (4.9,0.48) {$\bm{x}^*$};
	\end{tikzpicture}
	\caption{\label{fig:l1Algo} Illustration of Example~\ref{example:l1Algo}. The left subfigure depicts the sets $\Sup(1,\bm{\mathcal{S}},\bm{\mathcal{C}})$, $\Sup(2,\bm{\mathcal{S}},\bm{\mathcal{C}})$ and $\Sup(3,\bm{\mathcal{S}},\bm{\mathcal{C}})$. In addition the set $\Sup(\bm{\mathcal{S}},\bm{\mathcal{S}},\bm{\mathcal{C}})$ is highlighted with thick lines.
		The result of algorithm~\ref{alg:l1Algorithm} is shown in the right subfigure.}
\end{figure}

\section{Conclusions}\label{sec:conclusion}

\noindent In this paper, we have introduced a new complex system, composed of several multiobjective optimization problems, using a graph-based model. It was necessary to extend the classical concept of efficiency -~called superiority~- to consider only specific sets of objective functions, subsystem constraints, and linking constraints.

Furthermore, we proposed different methods and algorithms to solve the optimization problem associated with a complex system. We generalized the concept of the ideal point to obtain lower bounds. Hierarchical and scalarization based algorithms are elaborated to compute (weakly) superior solutions. In this context, we defined \(\varepsilon\)-superiority to approximate the set of superior solutions. Additionally, we proposed a concept of compromise and stated an algorithm to compute the set of compromise solutions. 

In the future it will be desirable to apply the presented theory and methodology to a real-world complex system.
\newline

\textbf{Acknowledgments.}
This work was supported by the bilateral cooperation project
\emph{Mathematical Modeling of Complex Systems} funded by the Deutscher Aka\-de\-mi\-scher Austauschdienst (DAAD, Project-ID 57211227), Deutsche Forschungs\-gemeinschaft (DFG, Project-ID RU 1524/2-3), Bundes\-ministerium f\"ur Bildung und Forschung (BMBF, Project-ID 13N14561) and by the project KoLBi (BMBF, Project-ID 01JA1507). The last co-author recognizes partial support from the United States Office of Naval Research through grant number N00014-16-1-2725.





\begin{thebibliography}{00}



\bibitem[Aonuma(1982)]{AonuArec1982}
T.~Aonuma.
\newblock A resource-directive decomposition algorithm for weakly coupled
dynamic linear programs.
\newblock \emph{Mathematische Operationsforschung und Statistik. Series
	Optimization}, 13\penalty0 (1):\penalty0 39--58, 1982.

\bibitem[Bard(1998)]{Bard1998}
J.~F. Bard.
\newblock \emph{Practical Bilevel Optimization: Algorithms and Applications}.
\newblock Kluwer, Dordrecht, 1998.

\bibitem[Bertsekas and Tsitsiklis(1997)]{BertPara1997}
D.~P. Bertsekas and J.~N. Tsitsiklis.
\newblock \emph{Parallel and distributed computation: numerical methods}.
\newblock Athena Scientific, Belmont, MA, 1997.

\bibitem[Ciucci et~al.(2012)Ciucci, Honda, and Yang]{CiucAnin2012}
F.~Ciucci, T.~Honda, and M.~C. Yang.
\newblock An information-passing strategy for achieving {P}areto optimality in
the design of complex systems.
\newblock \emph{Research in engineering design}, 23\penalty0 (1):\penalty0
71--83, 2012.

\bibitem[Dandurand and Wiecek(2015)]{DW:2015}
B.~Dandurand and M.~Wiecek.
\newblock Distributed computation of {P}areto sets.
\newblock \emph{SIAM Journal on Optimization}, 25\penalty0 (2):\penalty0
1083--1109, 2015.

\bibitem[Di~Matteo et~al.(2017)Di~Matteo, Dandy, and Maier]{diMaMult2017}
M.~Di~Matteo, G.~C. Dandy, and H.~R. Maier.
\newblock Multiobjective optimization of distributed stormwater harvesting
systems.
\newblock \emph{Journal of Water Resources Planning and Management},
143\penalty0 (6):\penalty0 1--13, 2017.

\bibitem[Ehrgott(2005)]{Ehrgott}
M.~Ehrgott.
\newblock \emph{Multicriteria Optimization}.
\newblock Springer, New York, 2005.

\bibitem[Ehtamo et~al.(2001)Ehtamo, Kettunen, and
H{\"a}m{\"a}l{\"a}inen]{EhtaSear2001}
H.~Ehtamo, E.~Kettunen, and R.~P. H{\"a}m{\"a}l{\"a}inen.
\newblock Searching for joint gains in multi-party negotiations.
\newblock \emph{European Journal of Operational Research}, 130:\penalty0
54--69, 2001.

\bibitem[Engau(2010)]{EngaInte2010}
A.~Engau.
\newblock Interactive decomposition-coordination methods for complex decision
problems.
\newblock In \emph{Handbook of Multicriteria Analysis}, pages 329--365.
Springer, 2010.

\bibitem[Fernandez and Olmedo(2013)]{FernAnou2013}
E.~Fernandez and R.~Olmedo.
\newblock An outranking-based general approach to solving group multi-objective
optimization problems.
\newblock \emph{European Journal of Operational Research}, 225\penalty0
(3):\penalty0 497--506, 2013.

\bibitem[Fulga(2007)]{FulgDece2007}
C.~Fulga.
\newblock Decentralized cooperative optimization for multi-criteria decision
making.
\newblock In \emph{Advances in Cooperative Control and Optimization}, pages
65--80. Springer, 2007.

\bibitem[Gardenghi et~al.(2011)Gardenghi, G{\'o}mez, Miguel, and
Wiecek]{gard:alge:2011}
M.~Gardenghi, T.~G{\'o}mez, F.~Miguel, and M.~Wiecek.
\newblock Algebra of efficient sets for multiobjective complex systems.
\newblock \emph{Journal of Optimization Theory and Applications}, 149:\penalty0
385--410, 2011.

\bibitem[Guarneri and Wiecek(2016)]{GW:2015}
P.~Guarneri and M.~Wiecek.
\newblock Pareto-based negotiation in distributed multidisciplinary design.
\newblock \emph{Structural and Multidisciplinary Optimization}, 53\penalty0
(4):\penalty0 657--671, 2016.

\bibitem[Haimes et~al.(1990)Haimes, Tarvainen, Shima, and
Thadathil]{HaimHier1990}
Y.~Haimes, K.~Tarvainen, T.~Shima, and J.~Thadathil.
\newblock Hierarchical multiobjective analysis of large-scale systems, 1990.

\bibitem[Heiskanen(1999)]{HeisDece1999}
P.~Heiskanen.
\newblock Decentralized method for computing {P}areto solutions in multiparty
negotiations.
\newblock \emph{European Journal of Operational Research}, 117:\penalty0
578--590, 1999.

\bibitem[Heiskanen et~al.(2001)Heiskanen, Ehtamo, and
H{\"a}m{\"a}l{\"a}inen]{HeisCons2001}
P.~Heiskanen, H.~Ehtamo, and R.~P. H{\"a}m{\"a}l{\"a}inen.
\newblock Constraint proposal method for computing {P}areto solutions in
multi-party negotiations.
\newblock \emph{European Journal of Operational Research}, 133:\penalty0
44--61, 2001.

\bibitem[Ji et~al.(2014)Ji, Liu, and Liao]{JiLiOnre2014}
L.~Ji, Q.~Liu, and X.~Liao.
\newblock On reaching group consensus for linearly coupled multi-agent
networks.
\newblock \emph{Information Sciences}, 287:\penalty0 1--12, 2014.

\bibitem[Juel(1980)]{Jue80}
R.~Juel, H.;~Love.
\newblock Hull {P}roperties in {L}ocation {P}roblems.
\newblock \emph{Working Paper No. 166; Faculty of Business McMaster University,
	Hamilton, Ontario}, 1980.

\bibitem[Kang et~al.(2014)Kang, Kokkolaras, and Papalambros]{Kangetal2014}
N.~Kang, M.~Kokkolaras, and P.~Y. Papalambros.
\newblock Solving multiobjective optimization problems using quasi-separable
mdo formulations and analytical target cascading.
\newblock \emph{Structural and Multidisciplinary Optimization}, 50:\penalty0
849--859, 2014.

\bibitem[Kim et~al.(2003)Kim, Rideout, Papalambros, and Stein]{Kimetal2003}
H.~M. Kim, D.~G. Rideout, P.~Y. Papalambros, and J.~L. Stein.
\newblock Analytical target cascading in automotive vehicle design.
\newblock \emph{Journal of Mechanical Design}, 125\penalty0 (3):\penalty0
481--489, 2003.

\bibitem[Konnov(2013)]{KonnVect2013}
I.~V. Konnov.
\newblock Vector network equilibrium problems with elastic demands.
\newblock \emph{Journal of Global Optimization}, 57\penalty0 (2):\penalty0
521--531, 2013.

\bibitem[Leverenz et~al.(2016)Leverenz, Xu, and Wiecek]{LeveMult2016}
J.~Leverenz, M.~Xu, and M.~M. Wiecek.
\newblock Multiparametric optimization for multidisciplinary engineering
design.
\newblock \emph{Structural and Multidisciplinary Optimization}, 54\penalty0
(4):\penalty0 795--810, 2016.

\bibitem[Lewis and Mistree(1997)]{lewis1997}
K.~Lewis and F.~Mistree.
\newblock Modeling interactions in multidisciplinary design: a game theoretic
approach.
\newblock \emph{{AIAA} Journal}, 35\penalty0 (8):\penalty0 1387--1392, 1997.

\bibitem[Li et~al.(2014)Li, Wu, and Hu]{LiWuDesi2014}
F.~Li, T.~Wu, and M.~Hu.
\newblock Design of a decentralized framework for collaborative product design
using memetic algorithms.
\newblock \emph{Optimization and Engineering}, 15\penalty0 (3):\penalty0
657--676, 2014.

\bibitem[Lou and Wang(2016)]{LouWAppr2016}
Y.~Lou and S.~Wang.
\newblock Approximate representation of the {P}areto frontier in multiparty
negotiations: Decentralized methods and privacy preservation.
\newblock \emph{European Journal of Operational Research}, 254\penalty0
(3):\penalty0 968--976, 2016.

\bibitem[Luca and Pillo(1987)]{LucaExac1987}
A.~d. Luca and G.~d. Pillo.
\newblock Exact augmented lagrangian approach to multilevel optimization of
large-scale systems.
\newblock \emph{International Journal of Systems Science}, 18\penalty0
(1):\penalty0 157--176, 1987.

\bibitem[Martins and Lambe(2013)]{MarLam2013}
J.~R. R.~A. Martins and A.~B. Lambe.
\newblock Multidisciplinary design optimization: A survey of architectures.
\newblock \emph{AIAA Journal}, 51\penalty0 (9):\penalty0 2049--2075, 2013.

\bibitem[Miguel et~al.(2009)Miguel, G{\'o}mez, Luque, Ruiz, and
Caballero]{MiguAdec2009}
F.~Miguel, T.~G{\'o}mez, M.~Luque, F.~Ruiz, and R.~Caballero.
\newblock A decomposition-coordination method for complex multi-objective
systems.
\newblock \emph{Asia-Pacific Journal of Operational Research}, 26\penalty0
(6):\penalty0 735--757, 2009.

\bibitem[Naz et~al.(2016)Naz, Mushtaq, Naeem, Iqbal, Altaf, and
Haneef]{NazMMult2016}
M.~N. Naz, M.~I. Mushtaq, M.~Naeem, M.~Iqbal, M.~W. Altaf, and M.~Haneef.
\newblock Multicriteria decision making for resource management in renewable
energy assisted microgrids.
\newblock \emph{Renewable and Sustainable Energy Reviews}, 71:\penalty0
323--341, 2016.

\bibitem[Nolan et~al.(2007)Nolan, Jin, and Chunhang]{Nolan2007}
P.~Nolan, Z.~Jin, and L.~Chunhang.
\newblock \emph{The Global Business Revolution and the Cascade Effect: Systems
	Integration in the Aerospace, Beverages and Retail Industries}.
\newblock Springer, 2007.

\bibitem[{Ringuest J.L.}(1992)]{RJ92}
{Ringuest J.L.}
\newblock Compromise {P}rogramming.
\newblock \emph{Multiobjective Optimization: Behavioral and Computational
	Considerations. Springer, Boston, MA}, 1992.

\bibitem[{Roland Durier}(1994)]{RD94}
{Roland Durier}.
\newblock Convex hull properties in location theory.
\newblock \emph{Numerical Functional Analysis and Optimization}, \penalty0
(15:5-6):\penalty0 567--582, 1994.

\bibitem[Roth(1985)]{RothSome1985}
A.~E. Roth.
\newblock Some additional thoughts on post-settlement settlements.
\newblock \emph{Negotiation Journal}, 1\penalty0 (3):\penalty0 245--247, 1985.

\bibitem[Shimizu and Aiyoshi(1981)]{ShimHier1981}
K.~Shimizu and E.~Aiyoshi.
\newblock Hierarchical multi-objective decision systems for general resource
allocation problems.
\newblock \emph{Journal of Optimization Theory and Applications}, 35\penalty0
(4):\penalty0 517--533, 1981.

\bibitem[Shimizu et~al.(1997)Shimizu, Ishizuka, and Bard]{ShimIshBa1997}
K.~Shimizu, Y.~Ishizuka, and J.~F. Bard.
\newblock \emph{Nondifferentiable and Two-Level Mathematical Programming}.
\newblock Kluwer, 1997.

\bibitem[Simpson(1750)]{Sim50}
T.~Simpson.
\newblock \emph{The {D}octrine and {A}pplication of {F}luxions}.
\newblock London, 1750.

\bibitem[Stummer and Vetschera(2003)]{StumDece2003}
C.~Stummer and R.~Vetschera.
\newblock Decentralized planning for multiobjective resource allocation and
project selection.
\newblock \emph{Central European Journal of Operations Research}, 11\penalty0
(3):\penalty0 253--279, 2003.

\bibitem[Teich et~al.(1996)Teich, Wallenius, Wallenius, and
Zionts]{TeicIden1996}
J.~E. Teich, H.~Wallenius, J.~Wallenius, and S.~Zionts.
\newblock Identifying {P}areto-optimal settlements for two-party resource
allocation negotiations.
\newblock \emph{European Journal of Operational Research}, 93\penalty0
(3):\penalty0 536--549, 1996.

\bibitem[Tosserams et~al.(2008)Tosserams, Etman, and Rooda]{TossAugm2008}
S.~Tosserams, L.~Etman, and J.~Rooda.
\newblock Augmented lagrangian coordination for distributed optimal design in
mdo.
\newblock \emph{International journal for numerical methods in engineering},
73\penalty0 (13):\penalty0 1885--1910, 2008.

\bibitem[Tosserams et~al.(2009)Tosserams, Etman, and Rooda]{TossBloc2009}
S.~Tosserams, L.~F. Etman, and J.~Rooda.
\newblock Block-separable linking constraints in augmented lagrangian
coordination.
\newblock \emph{Structural and Multidisciplinary Optimization}, 37\penalty0
(5):\penalty0 521--527, 2009.

\bibitem[Vanderplaats(2007)]{Vander2007}
G.~N. Vanderplaats.
\newblock \emph{Multidiscipline Design Optimization}.
\newblock Vanderplaats Research {\&} Development, Inc., 2007.

\end{thebibliography}


\end{document}